\documentclass[reqno, 10pt, centertags,draft]{amsart}
\usepackage{amsmath,amsthm,amscd,amssymb,latexsym,upref,stmaryrd,esint}

\makeatletter
\def\theequation{\@arabic\c@equation}

\newcommand{\bbN}{{\mathbb{N}}}
\newcommand{\bbR}{{\mathbb{R}}}

\newcommand{\bbZ}{{\mathbb{Z}}}
\newcommand{\bbC}{{\mathbb{C}}}

\newcommand{\bbJ}{{\mathbb{J}}}

\newcommand{\R}{\mathbb{R}}

\newcommand{\N}{\mathbb{N}}

\newcommand{\cB}{{\mathcal B}}

\newcommand{\cG}{{\mathcal G}}
\newcommand{\cH}{{\mathcal H}}

\newcommand{\cK}{{\mathcal K}}

\newcommand{\cM}{{\mathcal M}}

\newcommand{\cR}{{\mathcal R}}

\newcommand{\no}{\nonumber}
\newcommand{\lb}{\label}
\newcommand{\f}{\frac}

\newcommand{\ol}{\overline}

\newcommand{\Oh}{O}
\newcommand{\oh}{o}

\newcommand{\loc}{\text{\rm{loc}}}

\newcommand{\ran}{\text{\rm{ran}}}

\newcommand{\dom}{\text{\rm{dom}}}

\newcommand{\slimes}{\text{\rm{l.i.m.}}}

\newcommand{\bi}{\bibitem}

\DeclareMathOperator{\tr}{tr}

\allowdisplaybreaks
\numberwithin{equation}{section}

\newtheorem{theorem}{Theorem}[section]

\newtheorem{lemma}[theorem]{Lemma}
\newtheorem{corollary}[theorem]{Corollary}
\newtheorem{definition}[theorem]{Definition}

\theoremstyle{remark}
\newtheorem{remark}[theorem]{Remark}

\begin{document}

\title[Riesz Basis Criterion for Non-Self-Adjoint Schr\"odinger Operators]{A 
Schauder and Riesz Basis Criterion for Non-Self-Adjoint Schr\"odinger Operators with 
Periodic and Antiperiodic Boundary Conditions}
\author[F.\ Gesztesy and V.\ Tkachenko]{Fritz
Gesztesy and Vadim Tkachenko}
\address{Department of Mathematics,
University of Missouri, Columbia, MO 65211, USA}
\email{fritz@math.missouri.edu}
\urladdr{http://www.math.missouri.edu/personnel/faculty/gesztesyf.html}
\address{Department of Mathematics,
Ben Gurion University of the Negev, Beer--Sheva 84105, Israel}
\email{tkachenk@math.bgu.ac.il}
\date{\today}

\begin{abstract} 
Under the assumption that $V \in L^2([0,\pi]; dx)$, 
we derive necessary and sufficient conditions for (non-self-adjoint) 
Schr\"odinger operators $-d^2/dx^2+V$ in $L^2([0,\pi]; dx)$ with periodic and antiperiodic boundary conditions to possess a Riesz basis of root vectors (i.e., eigenvectors and generalized eigenvectors spanning the range of the Riesz projection associated with the corresponding periodic and antiperiodic eigenvalues). 

We also discuss the case of a Schauder basis for periodic and antiperiodic 
Schr\"odinger operators $-d^2/dx^2+V$ in $L^p([0,\pi]; dx)$, $p \in (1,\infty)$.
\end{abstract}

\maketitle

\section{Introduction} \lb{s1}

We study (generally, non-self-adjoint) Schr\"odinger operators $H^{P}$ 
and $H^{AP}$ in the Hilbert space $L^2([0,\pi]; dx)$ associated with the 
differential expression
\begin{equation}
L = - \f{d^2}{dx^2} + V(x), \quad x\in [0,\pi],    \lb{1.1}
\end{equation}
and complex-valued potential $V$ satisfying
\begin{equation}
V \in L^2([0,\pi]; dx),    \lb{1.2}
\end{equation}
with {\it periodic} and {\it antiperiodic} boundary conditions defined by 
\begin{align}
& (H^Pf)(x)=(L f)(x),  \quad x \in [0,\pi],    \no \\ 
& \, f\in\dom\big(H^P\big)=\big\{g\in L^2([0,\pi]; dx) \,\big|\, g,g'\in AC([0,\pi]); \,
L g \in L^2([0,\pi]; dx);    \lb{1.3} \\
& \hspace*{7.45cm} g(\pi)=g(0), \, g'(\pi)=g'(0)\big\}, 
\no
\end{align}
and 
\begin{align}
& (H^{AP}f)(x)=(L f)(x),  \quad x \in [0,\pi],    \no \\
& \, f\in\dom\big(H^{AP}\big)=\big\{g\in L^2([0,\pi]; dx) \,\big|\, g,g'\in AC([0,\pi]); \,
L g \in L^2([0,\pi]; dx);    \no \\
& \hspace*{6.3cm} g(\pi)=-g(0), \, g'(\pi)=-g'(0)\big\},     \lb{1.4} 
\end{align}
respectively. In addition to the periodic and antiperiodic 
Schr\"odinger operators $H^P$ and $H^{AP}$ we also invoke the corresponding Dirichlet operator $H^D$ in $L^2([0,\pi]; dx)$ defined by 
\begin{align}
& (H^Df)(x)=(L f)(x),  \quad x \in [0,\pi],    \no \\
& f\in\dom\big(H^D\big)=\big\{g\in L^2([0,\pi]; dx) \,\big|\, g,g'\in AC([0,\pi]); \,
L g \in L^2([0,\pi]; dx);     \lb{1.5}  \\
& \hspace*{8.9cm} g(0)=g(\pi)=0\big\}.    \no
\end{align}

On occasion we shall also mention the Neumann operator $H^N$ 
in $L^2([0,\pi]; dx)$ defined by 
\begin{align}
& (H^Nf)(x)=(L f)(x),  \quad x \in [0,\pi],    \no \\
& f\in\dom\big(H^D\big)=\big\{g\in L^2([0,\pi]; dx) \,\big|\, g,g'\in AC([0,\pi]); \,
L g \in L^2([0,\pi]; dx);     \lb{1.5a}  \\
& \hspace*{8.75cm} g'(0)=g'(\pi)=0\big\}.    \no
\end{align}

One notes that $H^P$, $H^{AP}$, $H^D$, and $H^N$ are closed and densely defined; they are self-adjoint if and only if $V$ is real-valued a.e.\ on $[0,\pi]$. 
In particular, the boundary conditions in $H^P$, $H^{AP}$, $H^D$, and $H^N$ 
are all self-adjoint.

The Schr\"odinger equation associated with $L$ will be written in the form 
\begin{equation}
L \psi(\zeta,x) = \zeta^2 \psi(\zeta,x), \quad \zeta\in \bbC, \; x \in [0,\pi],    \lb{1.6} 
\end{equation}
with $\psi, \psi' \in AC([0,\pi])$ (the set of absolutely continuous functions on 
$[0,\pi]$). Moreover, we emphasize the notational choice in \eqref{1.6} which 
depicts $z = \zeta^2$ as the energy variable (i.e., $\zeta$ as momentum), which will be 
convenient in the following. 

The spectra of $H^P$, $H^{AP}$, and $H^D$ are well-known to be purely 
discrete (cf., e.g., \cite[Sect.\ 1.3]{Ma86}), that is, their resolvents are compact,
\begin{equation}
\big(H^{Q} -z I\big)^{-1} \in \cB_{\infty}\big(L^2([0,\pi]; dx)\big), 
\quad z \in \rho\big(H^{Q}\big),    \lb{1.7}
\end{equation}
where $Q$ stands for $P$, $AP$, and $D$. In fact, the resolvents in 
\eqref{1.7} are known to lie in the trace class $\cB_1\big(L^2([0,\pi]; dx)\big)$. 
In particular, we will use the notation
\begin{align}
& \sigma\big(H^P\big) = \sigma_d \big(H^P\big) 
= \{\lambda_0^+, \lambda_{2k}^+, \, \lambda_{2k}^-\}_{k\in\bbN},     
\lb{1.7a} \\
& \sigma\big(H^{AP}\big) = \sigma_d \big(H^{AP}\big) 
= \{\lambda_{2k+1}^+, \, \lambda_{2k+1}^-\}_{k\in\bbN_0}, 
\lb{1.7b} \\
& \sigma\big(H^D\big) = \sigma_d \big(H^D\big) 
= \{\mu_j\}_{j\in\bbN},    \lb{1.7c}
\end{align}
for the spectra of $H^P$, $H^{AP}$, and $H^D$ (where $\bbN_0 = \bbN \cup\{0\}$).

The asymptotic behavior of the eigenvalues of $H^P$, $H^{AP}$, and $H^D$ can 
be described as follows: Let $k \in \bbN$ and $k \geq k_0$, for some 
$k_0 \in \bbN$ sufficiently large, then there exists $R = R(k_0) > 0$ such 
that every open disk $D_{2k}(R)=\{z\in\bbC \,|\, |z-4k^2| < R\} \subset \bbC$, 
$k\in\bbN$, with center at $4k^2$ contains precisely two eigenvalues of 
$H^{P}$ denoted by 
\begin{equation}
\lambda_{2k}^+, \, \lambda_{2k}^-, \; k\in\bbN,  \lb{1.8}
\end{equation}
and every open disk 
$D_{2k+1}(R)=\{z\in\bbC \,|\, |z-(2k+1)^2| < R\} \subset \bbC$, 
$k\in\bbN$, with center at $(2k+1)^2$ contains precisely two eigenvalues 
of $H^{AP}$ denoted by 
\begin{equation}
\lambda_{2k+1}^+, \, \lambda_{2k+1}^-, \; k\in\bbN.  \lb{1.9}
\end{equation}
In addition, every open disk 
$D_{j}(R)=\{z\in\bbC \,|\, |z-j^2| < R\} \subset \bbC$, 
$j\in\bbN$, with center at $j^2$ contains precisely one eigenvalue 
of $H^{D}$ (resp., $H^N$) denoted by 
\begin{equation}
\mu_{j} \, \text{ (resp., $\nu_j$), } \; j\in\bbN.    \lb{1.10}
\end{equation}

Next, we recall a few facts regarding the eigenvalues of a compact, linear 
operator  $T \in \cB_{\infty}(\cH)$ in a separable complex Hilbert space $\cH$: 
The {\it geometric multiplicity}, $m_g(\lambda_0,T)$, of an eigenvalue 
$\lambda_0 \in \sigma_p (T)$ of $T$ is given by 
\begin{equation}
m_g(\lambda_0,T) = \dim(\ker(T - \lambda_0 I_{\cH})),    \lb{1.11} 
\end{equation}
with $\ker(T - \lambda_0 I_{\cH})$ a closed linear subspace in $\cH$. 

The set of all {\it root vectors} of $T$ (i.e., eigenvectors and generalized eigenvectors, or associated eigenvectors) corresponding to 
$\lambda_0 \in \sigma_p (T)$ is given by 
\begin{equation} 
\cR(\lambda_0,T) = \big\{f\in\cH\,\big|\, (T - \lambda_0 I_{\cH})^k f = 0 \; 
\text{for some $k\in\bbN$}\big\}.     \lb{1.12}
\end{equation}
For $\lambda_0 \in \sigma_p (T) \backslash \{0\}$, the set 
$\cR(\lambda_0,T)$ is a closed linear subspace of $\cH$ whose dimension 
equals the {\it algebraic multiplicity}, $m_a(\lambda_0,T)$, of $\lambda_0$, 
\begin{equation}
m_a(\lambda_0,T) = \dim\big(\big\{f\in\cH\,\big|\, (T - \lambda_0 I_{\cH})^k f = 0 
\; \text{for some $k\in\bbN$}\big\}\big)<\infty.   \lb{1.13}
\end{equation}
One has  
\begin{equation} 
m_g(\lambda_0,T) \leq m_a(\lambda_0,T), \quad 
\lambda_0 \in \sigma_p (T) \backslash \{0\}.    \lb{1.14}
\end{equation}

Moreover, for $\lambda_0 \in \sigma_p (T) \backslash \{0\}$ one can introduce 
the Riesz projection, $P(\lambda_0,T)$ of $T$ corresponding to $\lambda_0$, by 
\begin{equation}
P(\lambda_0,T)=-\f{1}{2\pi i}\ointctrclockwise_{C(\lambda_0; \varepsilon) }
d\zeta \, (T - \zeta I_{\cH})^{-1}, \lb{1.15}
\end{equation}
with $C(\lambda_0; \varepsilon) $ a counterclockwise oriented circle centered at 
$\lambda_0$ with sufficiently small radius $\varepsilon>0$, such that the closed 
disk with center $\lambda_0$ and radius $\varepsilon$ excludes   
$\sigma(T)\backslash\{\lambda_0\}$. In this case one has (cf.\ \cite[Sect.\ II.1]{GGK90}, \cite{GK60}, \cite[Sects.\ I.1, I.2]{GK69})
\begin{equation}
m_a(\lambda_0,T) = \dim(\ran(P(\lambda_0,T)) < \infty, \quad 
\cR(\lambda_0,T) = \ran(P(\lambda_0,T)).      \lb{1.16}
\end{equation}

We are particularly interested in the case where $A$ is a densely defined, 
closed, linear operator in $\cH$ whose resolvent is compact, that is, 
\begin{equation}
(A - z I_{\cH})^{-1} \in \cB_{\infty}(\cH), \quad z \in \rho (A).    \lb{1.17}
\end{equation}
Via the spectral mapping theorem all eigenvalues of $A$ then correspond to nonzero eigenvalues of its compact resolvent $(A - z I_{\cH})^{-1}$, 
$z \in \rho (A)$, and vice versa. Hence, we use the same notions of root 
vectors, root spaces, and geometric and algebraic multiplicities associated with the 
eigenvalues of $A$. Moreover, in the case where 
\begin{equation}
(A - z I_{\cH})^{-1} \in \cB_1 (\cH), \quad z \in \rho (A),    \lb{1.17a}
\end{equation}
we recall the fact that (cf.\ \cite{GW95})
\begin{align}
\begin{split} 
& {\det}_{\cH} \big(I_{\cH} - (z-z_0) (A - z I_{\cH})^{-1}\big) 
= \prod_{j\in \bbJ} \bigg(\f{\lambda_j - z}{\lambda_j - z_0}
\bigg)^{m_a(\lambda_j,A)}   \\
& \quad \underset{z \to \lambda_k}{=} 
(\lambda_k - z)^{m_a(\lambda_k,A)} [C_k + \Oh(z - \lambda_k)],  
\quad C_k \neq 0, \; k \in \bbJ,     \lb{1.17b}
\end{split} 
\end{align}
where $\bbJ \subseteq \bbN$ denotes an appropriate index set such that 
$\sigma (A) = \{\lambda_j\}_{j \in \bbJ}$ with $\lambda_j \neq \lambda_{j'}$, 
$j \neq j'$, $j, j' \in \bbJ$.
 \medskip
 
We recall that a system of vectors $\{g_k\}_{k\in\bbN} \subset \cH$ is called 
{\it complete} in $\cH$ if $\big(\{g_k\}_{k\in\bbN}\big)^\bot = \{0\}$ (equivalently, 
if $\ol{\text{lin.\,span} \, \{g_k\}_{k \in \bbN}} = \cH$). Moreover, 
a system $\{h_k\}_{k \in \bbN} \subset \cH$ is called {\it minimal} if no vector 
$h_{k_0} \in \{h_k\}_{k \in \bbN}$ satisfies 
$h_{k_0} \in \ol{\text{lin.\,span} \, \{h_k\}_{k \in \bbN\backslash \{k_0\}}}$. The 
system $\{f_k\}_{k\in\bbN} \subset \cH$ is called a {\it Schauder basis} in $\cH$
\begin{align}
\begin{split}
& \text{if for each $f \in \cH$, there exist unique $c_k = c_k(f) \in\bbC$, $k \in \bbN$, 
such that}     \lb{1.17c} \\
& \quad \text{$f = \sum_{k \in \bbN} c_k(f) f_k$ converges in the norm of $\cH$.} 
\end{split} 
\end{align} 
Two systems 
$\{g_k\}_{k\in\bbN} \subset \cH$ and $\{h_k\}_{k \in \bbN} \subset \cH$ are 
called {\it biorthogonal} if 
\begin{equation}
(g_j, h_k)_{\cH} = \delta_{j,k}, \quad j,k \in \bbN.    \lb{1.17d}
\end{equation}
We also recall (cf.\ \cite[Sect.\ VI.1]{GK69}) that if 
$\{g_k\}_{k\in\bbN} \subset \cH$ is minimal, then a biorthogonal system 
$\{h_k\}_{k \in \bbN} \subset \cH$ exists. Moreover, if $\{g_k\}_{k\in\bbN}$ is minimal, then the biorthogonal system $\{h_k\}_{k \in \bbN}$ in $\cH$ is 
uniquely determined if and only if $\{g_k\}_{k\in\bbN}$ is complete in $\cH$. 

Next, we turn to the definition of a Riesz basis in a Hilbert space due to Bari (cf., e.g., \cite[Sect.\ VI.2]{GK69}, \cite[Sect.\ 1.8]{Yo01}):

\begin{definition} \lb{d1.1}
Let $\cH$ be a complex separable Hilbert space and $f_k \in \cH$, $k\in \bbN$. Then the system 
$\{f_k\}_{k\in\bbN}$ is called a Riesz basis in $\cH$ if there exists an operator 
$A \in \cB(\cH)$, with $A^{-1} \in \cB(\cH)$, and an orthonormal system 
$\{e_k\}_{k\in\bbN}$ in $\cH$, such that
\begin{equation}
f_k = A e_k, \quad k \in \bbN.    \lb{1.18}
\end{equation} 
\end{definition}

Given these preparations, we can now formulate the principal result of this paper:

\begin{theorem} \lb{t1.2}
Assume $V \in L^2([0,\pi]; dx)$, then the following results hold: \\ 
$(i)$ The system of root vectors of $H^P$ contains a 
Riesz basis in $L^2([0,\pi]; dx)$ if and only if 
\begin{equation}
\sup_{\substack{k \in \bbN, \\ \lambda_{2k}^+ \neq \lambda_{2k}^-}} 
\f{|\mu_{2k} - \lambda_{2k}^{\pm}|}{|\lambda_{2k}^+ - \lambda_{2k}^-|} < \infty.   
\lb{1.19}
\end{equation}
$(ii)$ The system of root vectors of $H^{AP}$ contains a 
Riesz basis in $L^2([0,\pi]; dx)$ if and only if 
\begin{equation}
\sup_{\substack{k \in \bbN, \\ \lambda_{2k+1}^+ \neq \lambda_{2k+1}^-}} 
\f{|\mu_{2k+1} - \lambda_{2k+1}^{\pm}|}{|\lambda_{2k+1}^+ - \lambda_{2k+1}^-|} < \infty.  \lb{1.20}
\end{equation}
\end{theorem} 

Here $\sup_{k \in \bbN, \, \lambda_{j}^+ \neq \lambda_{j}^-}$ signifies that all subscripts $j\in\bbN$ in \eqref{1.19} and \eqref{1.20} for which 
$\lambda_{j}^+$ and $\lambda_{j}^-$ coincide are simply excluded from 
the supremum considered. 

\begin{remark} \lb{r1.3} 
$(i)$ It is remarkable that only simple periodic (resp., 
antiperiodic) eigenvalues enter in the necessary and sufficient conditions 
\eqref{1.19} (resp., \eqref{1.20}) for the existence of a Riesz basis of root vectors 
of $H^P$ (resp., $H^{AP}$). The 
multiple periodic (resp., antiperiodic) eigenvalues play no role in deciding whether or not the system of root vectors of $H^P$ (resp., $H^{AP}$) constitutes 
a Riesz basis in $L^2([0,\pi]; dx)$. This leads to an interesting class of examples 
exhibiting a Riesz basis of root vectors as discussed in Remark \ref{r6.4}. \\
$(ii)$ In addition, it is remarkable that only every other Dirichlet eigenvalue (i.e., 
half the Dirichlet spectrum) enter the criterion \eqref{1.19} (resp., \eqref{1.20}). \\
$(iii)$ Of course, an analogous result holds with the Dirichlet spectrum 
$\sigma\big(H^D\big) = \{\mu_j\}_{j\in\bbN}$ replaced by the Neumann spectrum
$\sigma\big(H^N\big)= \{\nu_j\}_{j\in\bbN_0}$. 
\end{remark}

Next, we briefly turn to the closed Schr\"odinger operators $H^P$, $H^{AP}$, and 
$H^D$ generated by the differential expression $L = - (d^2/dx^2) + V(x)$ in $L^p([0,1]; dx)$, 
replacing the Hilbert space $L^2([0,1]; dx)$ by the Banach space 
$L^p([0,1]; dx)$, $p \in (1,\infty)$, in \eqref{1.3}, \eqref{1.4}, and \eqref{1.5}. (For notational simplicity we keep the same symbols for these three operators as well as for their eigenvalues in the $L^p$-context.)

\begin{theorem} \lb{t1.4}
Assume $V \in L^2([0,\pi]; dx)$ and let $p \in (1,\infty)$. Then the following results hold: \\ 
$(i)$ The system of root vectors of $H^P$ contains a 
Schauder basis in $L^p([0,\pi]; dx)$ if and only if \eqref{1.19} holds. \\
$(ii)$ The system of root vectors of $H^{AP}$ contains a 
Schauder basis in $L^p([0,\pi]; dx)$ if and only if \eqref{1.20} holds. 
\end{theorem} 

\begin{corollary} \lb{c1.5}
Assume $V \in L^2([0,\pi]; dx)$. Then the system of root vectors of $H^P$ 
$($resp., $H^{AP}$$)$ contains a Riesz basis in $L^2([0,\pi]; dx)$ if and only if 
the system of root vectors of $H^P$ $($resp., $H^{AP}$$)$ contains a Schauder 
basis in $L^2([0,\pi]; dx)$. 
\end{corollary}

Since the proofs for the cases of $H^P$ and $H^{AP}$ are completely 
analogous, we will focus exclusively on the case of $H^P$ in the remainder 
of this paper.

To properly place Theorem \ref{t1.2} in perspective, 
we also recall the following more general family of densely defined, closed, linear operators $H(t)$, $t\in [0, 2\pi)$ in $L^2([0,\pi]; dx)$ given by
\begin{align}
& (H(t)f)(x)=(L f)(x),  \quad x \in [0,\pi],     \no \\
& \, f\in\dom(H(t))=\big\{g\in L^2([0,\pi]; dx) \,\big|\, g,g'\in AC([0,\pi]); \,
L g \in L^2([0,\pi]; dx);     \no \\
& \hspace*{4cm} g(\pi)=e^{it}g(0), \, g'(\pi)=e^{it}g'(0)\big\}, \quad 
t\in [0, 2\pi).       \lb{1.22}
\end{align}
The family $\{H(t)\}_{t\in [0, 2\pi)}$ is of fundamental importance in the spectral theory of periodic Schr\"odinger operators in $L^2(\bbR; dx)$ (see below), and again, $H(t)$, $t\in [0,2\pi)$, is self-adjoint if and only if $V$ is real-valued 
a.e.\ on $[0,\pi]$. Moreover, 
one recalls that for $t \in (0,2 \pi)\backslash \{\pi\}$, that is, for all boundary conditions {\it different} from periodic and antiperiodic ones, the $t$-dependent 
boundary conditions in \eqref{1.22} are strongly Birkhoff regular and hence 
$H(t)$, $t \in (0,2 \pi)\backslash \{\pi\}$, possesses a Riesz basis of root vectors  
(cf.\ \cite{Ke64}, \cite{Mi62}). For this reason we deal exclusively with the special 
periodic ($t=0$) and antiperiodic ($t=\pi$) cases in this paper which are Birkhoff 
regular, but not strongly regular. 

Next, to provide a further perspective, we briefly describe the connection 
between the family of operators $H(t)$, $t \in [0, 2\pi)$, in $L^2([0,\pi]; dx)$, 
and non-self-adjoint Hill operators $H$ in $L^2(\bbR; dx)$, that is, periodic 
Schr\"odinger operators on $\bbR$ defined by
\begin{align}
\begin{split}
& (Hf)(x)=(L f)(x),  \quad x\in\bbR, \\
& f\in\dom(H)=\{g\in L^2(\bbR) \,|\, g,g'\in AC_{\loc}(\bbR); \,
L g \in L^2(\bbR)\}.   \lb{1.25}
\end{split}
\end{align}
with $L = - \f{d^2}{dx^2} + V(x)$, $x\in\bbR$, and $V\in L^2_{\loc}(\bbR)$ 
periodic on $\bbR$ with period $\pi$, 
\begin{equation}
V(x+\pi)=V(x) \, \text{ for a.e.\ $x\in\bbR$.}     \lb{1.26}
\end{equation}

The connection between $H$ in $L^2(\bbR; dx)$ and that of the family 
$H(t)$, $t\in [0, 2\pi)$ in $L^2([0,\pi]; dx)$ is then explicitly given by 
\begin{equation}
\cG H {\cG}^{-1} = \f{1}{2\pi}\int^{\oplus}_{[0,2\pi]} dt\, H(t).     \lb{1.27}
\end{equation}
Here we introduced the Gelfand transform (\cite{Ge50} ) defined by 
\begin{equation}
\cG\colon \begin{cases} L^2(\bbR) \to \cK  \\
\hspace*{.75cm} f \mapsto (\cG f)(x,t)=F(x,t)=\slimes_{N \uparrow \infty}
\sum_{n=-N}^N  f(x+n \pi) e^{-int},   \\[1mm]
\end{cases}    \lb{1.28}  
\end{equation}
where $\cK$ abbreviates the direct integral of Hilbert spaces with constant 
fibers $L^2([0,2\pi])$, 
\begin{equation}
\cK=\f{1}{2\pi}\int^{\oplus}_{[0,2\pi]} dt\, L^2([0,2\pi]) 
= L^2\big([0,2\pi]; dt/(2\pi); L^2([0,\pi];dx)\big),     \lb{1.29}
\end{equation}
and where $\slimes$ denotes the limit in $\cK$. By inspection,
$\cG$ is a unitary operator. The inverse Gelfand transform is given by
\begin{equation}
{\cG}^{-1}=\begin{cases} \cK \to L^2(\bbR) \\
\hspace*{.05cm}
F \mapsto ({\cG}^{-1}F)(x+n \pi) = \f{1}{2\pi}\int_{[0,2\pi]}
dt\, F(x,t) e^{int}, \quad n\in\bbZ.  \\[1mm]
\end{cases}     \label{1.30}
\end{equation}
One then has the spectral connection
\begin{equation}
\sigma(H)=\bigcup_{0\leq t\leq
\pi}\sigma(H(t)) \lb{1.31}
\end{equation}
(we recall that $\sigma(H(2\pi -t))=\sigma(H(t))$, $t\in [0,\pi]$). 

There has been a flurry of activity in connection with spectral theory for 
non-self-adjoint Hill operators$H$ (and $H^P$, $H^{AP}$) and more general, 
for closed realizations of $L$ on $[0,\pi]$ in \eqref{1.1} for various 
(non-self-adjoint boundary conditions) especially since the late 1970's, 
some of which were inspired by connections to the Korteweg--deVries hierarchy of evolution equations (we refer, e.g., to \cite{BG06}--\cite{Ch06}, 
\cite{Ga80}--\cite{GW98a}, \cite{GU83}, \cite{Ko97}, \cite{Me77}, 
\cite{PT88}--\cite{PT91a}, \cite{Ro63}, \cite{ST96}--\cite{Sh04a}, 
\cite{Tk64}--\cite{Tk02}, \cite{Ve80}--\cite{Ve06}, 
\cite{We98}, \cite{We98a},  \cite{Zh69}). More closely related to the current 
paper are the following publications studying non-self-adjoint boundary value 
problems and associated completeness and basis properties of the associated 
root vector systems such as \cite{DV05}--\cite{Ef07},  
\cite{KM98}--\cite{Ki09}, \cite{Ma06}--\cite{MM08}, 
\cite{MM10}--\cite{Mi10}, \cite{SS07}--\cite{SV09}, \cite{Ve10}--\cite{VK07}. 
We also note that basis properties for (anti)periodic boundary conditions 
in the context of $L^p([0,\pi]; dx)$, $p \in (1,\infty)$, were recently discussed 
in \cite{MM10} (see also \cite{Ma06}, \cite{Ma06a}, \cite{Ma08}, \cite{Ma08a} 
for various other boundary conditions). 

Still, to the best of our knowledge, results of the type Theorem \ref{t1.2} 
appear to be without precedent in the literature. 

Finally, we briefly summarize the principal notation used in this manuscript: 
Let $\cH$ be a separable complex (infinite-dimensional) Hilbert space, 
$(\cdot,\cdot)_{\cH}$ the 
scalar product in $\cH$ (linear in the second factor), and $I_{\cH}$ the identity operator in $\cH$. Next, let $T$ be a linear operator in $\cH$, with $\dom(T)$, 
$\ran(T)$, and $\ker(T)$ denoting the
domain, range, and kernel (i.e., null space) of $T$. 
The spectrum, 
point spectrum, 
and resolvent set of a closed linear operator 
in $\cH$ will be denoted by 
$\sigma(\cdot)$, 
$\sigma_{\rm p}(\cdot)$, 
and $\rho(\cdot)$, respectively. The Banach space of bounded linear operators 
on $\cH$ is denoted by $\cB(\cH)$; 
similarly, $\cB_{\infty}(\cH)$ and $\cB_p(\cH)$, 
$p>0$, denote the Banach space of compact operators on $\cH$ and 
the $\ell^p(\bbN)$-based trace ideals of $\cB(\cH)$. 

Prime ${}^{\prime}$ denotes differentiation w.r.t. $x$, $\bullet$ abbreviates 
$\zeta$-derivatives where $z=\zeta^2$ represents the complex spectral 
parameter. We also use the abbreviation $\bbN_0=\bbN\cup\{0\}$. For 
simplicity of notation we will abbreviate $I = I_{L^2([0,\pi]; dx)}$.

\section{Floquet Theoretic Preliminaries} \lb{s2}

In this section we briefly recall some standard results on the operators 
$H^P$ and $H^{AP}$ and some elements of the associated  Floquet 
theory. The results recorded in this section were discussed in detail in 
\cite[Sects.\ 2--4]{GT09}, hence we reproduce them here 
for the convenience of the reader without proofs. 

For the remainder of this paper we assume condition \eqref{1.2}, that is, we 
will always suppose that $V\in L^2 ([0,\pi]; dx)$. 

Associated with the differential expression $L = - d^2/dx^2 + V(x)$, 
$x\in [0,\pi]$, one introduces the fundamental system of
distributional solutions $c(\zeta,\cdot)$ and $s(\zeta,\cdot)$ of
\begin{equation}
L \psi(\zeta,x) = \zeta^2 \psi(\zeta,x), \quad \zeta\in\bbC, \; x \in [0,\pi],  \lb{2.0}  
\end{equation} 
satisfying
\begin{equation}
c(\zeta,0)=s'(\zeta,0)=1, \quad
c'(\zeta,0)=s(\zeta,0)=0, \quad \zeta \in \bbC.  \lb{2.1}
\end{equation}
For each $x\in\bbR$, $c(\zeta,x)$ and $s(\zeta,x)$ are entire with respect
to $\zeta$. The monodromy matrix
$\cM(\zeta)$ is then given by
\begin{equation}
\cM(\zeta)=\begin{pmatrix} c(\zeta,\pi) & s(\zeta,\pi) \\
c'(\zeta,\pi) & s'(\zeta,\pi) \end{pmatrix}, \quad
\zeta \in\bbC,   \lb{2.2}
\end{equation}
and its eigenvalues $\rho_\pm(\zeta)$, the Floquet multipliers, satisfy
\begin{equation}
\rho_+(\zeta)\rho_-(\zeta)=1  \lb{2.3}
\end{equation}
since
\begin{equation}
\det(\cM(\zeta))=c(\zeta,\pi) s'(\zeta,\pi) - c'(\zeta,\pi) s(\zeta,\pi) = 1.
\lb{2.4}
\end{equation}
The Floquet discriminant $u_+(\cdot)$ is
then defined by
\begin{equation}
u_+(\zeta)=\tr(\cM(\zeta))/2=
[c(\zeta,\pi) + s'(\zeta,\pi)]/2, \quad \zeta\in\bbC,  \lb{2.5}
\end{equation}
$u_+(\cdot)$ is an even function of exponential type, and one obtains
\begin{equation}
\rho_\pm (\zeta) = u_+(\zeta)\pm i\sqrt{1-u_+(\zeta)^2},   \lb{2.6}
\end{equation}
with an appropriate choice of the square root branches. We also note that
\begin{equation}
|\rho_\pm(\zeta)|=1 \, \text{ if and only if } \, u_+(\zeta)\in [-1,1].
\lb{2.7}
\end{equation}
The significance of the Floquet discriminant is underscored by the fact that
\begin{align}
\sigma\big(H^P\big) &= \big\{\zeta^2 \in \bbC \,\big|\, u_+(\zeta) = 1\big\}, 
\lb{2.8} \\
\sigma\big(H^{AP}\big) &= \big\{\zeta^2 \in \bbC \,\big|\, u_+(\zeta) 
= - 1\big\},     \lb{2.9} \\
\sigma(H(t)) &= \big\{\zeta^2 \in \bbC \,\big|\, u_+(\zeta) 
= \cos(t)\big\}, \quad t \in [0,2\pi).    \lb{2.9a} 
\end{align}

Concerning the asymptotic behavior of $c(\zeta,\cdot)$, $s(\zeta,\cdot)$, 
and $u_+(\zeta)$ as $|\zeta|\to \infty$ we recall the representations 
\begin{align}
c(\zeta,x) &= \cos(\zeta x) + c \, \f{\sin(\zeta x)}{\zeta} + \f{f(\zeta,x)}{\zeta}, 
\lb{2.10} \\
s(\zeta,x) &= \f{\sin(\zeta x)}{\zeta} - c \, \f{\cos(\zeta x)}{\zeta^2} 
+ \f{g(\zeta,x)}{\zeta^2},      \lb{2.11} \\
u_+(\zeta) &= \cos(\pi(\zeta - (c/\zeta))) + \f{h(\zeta)}{\zeta^2},    \lb{2.12}
\end{align}
where $c \in \bbC$ is a constant, $f(\zeta,x)$, $g(\zeta,x)$, and $h(\zeta)$ are 
entire functions of exponential type with respect to $\zeta$ (the type of 
$h(\cdot)$ being less or equal to $\pi$) satisfying 
\begin{align}
& \int_{\bbR} d p \int_0^{\pi} dx \, |f(p,x)|^2 + 
\int_{\bbR} d p \int_0^{\pi} dx \, |g(p,x)|^2 < \infty,    \lb{2.13} \\
& \int_{\bbR} d p \, |h(p)|^2 < \infty, \quad 
\sum_{k \in \bbZ} |h(k)| < \infty.     \lb{2.14} 
\end{align}

The asymptotic representations \eqref{2.10}--\eqref{2.14} imply that every disk 
$D_{2k} = \{z\in\bbC \,|\, |z- 4 k^2| \leq 2|c| + 1\}$ for $k\in\bbN$ sufficiently large, 
contains precisely two eigenvalues $\lambda_{2k}^+$ and $\lambda_{2k}^-$ of $H^P$ and one eigenvalue $\mu_{2k}$ of $H^D$ satisfying 
\begin{align}
& \lambda_{2k}^{\pm} = \bigg[2 k + \f{c}{2k\pi} + \f{h_k^{\pm}}{k}\bigg]^2, \quad 
\sum_{k\in\bbN} |h_k^{\pm}|^2 < \infty,     \lb{2.15} \\
& \mu_{2k} = \bigg[2 k + \f{c}{2k\pi} + \f{g_k}{k}\bigg]^2, \quad 
\sum_{k\in\bbN} |g_k|^2 < \infty   \lb{2.16} 
\end{align}  
(cf.\ also \cite[Ch.\ 1, Sect.\ 3.4]{Ma86}). Similarly, every disk 
$D_{k} = \{z\in\bbC \,|\, |z- k^2| \leq 2|c| + 1\}$ contains 
for $k\in\bbN$ sufficiently large a 
critical point $\kappa_{k}$ of $\Delta_+(\zeta^2) = u_+(\zeta)$, that is, 
$u_+^{\bullet}\big(\kappa_{k}^{1/2}\big) = 0$, and \eqref{2.15} implies 
\begin{equation}
2|\lambda_{2k}^{\pm} - \kappa_{2k}| \underset{k \to \infty}{=} 
|\lambda_{2k}^+ - \lambda_{2k}^-| [1 + \oh(1)].    \lb{2.17}
\end{equation}

Moreover, substituting \eqref{2.15} into \eqref{2.10}, \eqref{2.11} yields the existence of a constant $C>0$ such that 
\begin{align}
\begin{split} 
& \sum_{k\in\bbN} \left[\bigg|\int_0^{\pi} dx \, 
c\big((\lambda_{2k}^{\pm})^{1/2},x\big) f(x)\bigg|^2 + 
|\lambda_{2k}^{\pm}| \, \bigg|\int_0^{\pi} dx \, 
s\big((\lambda_{2k}^{\pm})^{1/2},x\big) f(x)\bigg|^2\right]    \\
& \quad \leq C \| f \|_{L^2([0,\pi]; dx)}^2,   \quad 
f \in L^2([0,\pi]; dx).     \lb{2.18}
\end{split}
\end{align}

For future purposes it is also convenient to introduce
\begin{equation}
u_-(\zeta)=[c(\zeta,\pi)-s'(\zeta,\pi)]/2, \quad \zeta\in\bbC. \lb{2.19}
\end{equation}
Floquet solutions $\psi_\pm(\zeta,x)$, $x\in \bbR$, of 
$L \psi(\zeta,x)=\zeta^2 \psi(\zeta,x)$,  
normalized at $x=0$, associated with $L$ are then given by 
($\zeta^2\in\bbC\backslash\{\mu_j\}_{j\in\bbN}$)
\begin{align}
\psi_\pm(\zeta,x)&=c(\zeta,x)+[\rho_\pm(\zeta)-c(\zeta,\pi)]
s(\zeta,\pi)^{-1} s(\zeta,x)    \lb{2.19a}\\
& = c(\zeta,x)+m_\pm(\zeta) s(\zeta,x),    \lb{2.19b} \\
m_\pm(\zeta) &= \big[-u_-(\zeta)\pm i\sqrt{1-u_+(\zeta)^2}\big]
\big/s(\zeta,\pi),    \lb{2.20} \\
\rho_{\pm} (\zeta) &= u_+(\zeta) \pm i \sqrt{1-u_+(\zeta)^2},   
\lb{2.20a} \\
\psi_\pm(\zeta,0)&=1, \quad \psi_{\pm}'(\zeta,0) = m_{\pm} (\zeta).  \lb{2.20b}
\end{align}
One then verifies (for $\zeta^2\in\bbC\backslash\{\mu_j^2\}_{j\in\bbN}$, 
$x\in\bbR$),
\begin{align}
&\psi_{\pm}(\zeta,x+\pi)=\rho_\pm(\zeta) \psi_\pm(\zeta,x),   \lb{2.21} \\
& W(\psi_+(\zeta,\cdot),\psi_-(\zeta,\cdot))= m_-(\zeta)- m_+(\zeta)= -
2i\sqrt{1-u_+ (\zeta)^2} / s(\zeta,\pi),  \lb{2.22} \\
& m_+(\zeta)+m_-(\zeta)=-2u_-(\zeta)/s(\zeta,\pi),  \lb{2.23} \\
& m_+(\zeta)m_-(\zeta)=- c'(\zeta,\pi)/s(\zeta,\pi), \lb{2.24}
\end{align}
where 
\begin{equation}
W(f,g)(x)=f(x)g'(x)-f'(x)g(x),  \quad x\in [0,\pi] \, \text{ (resp., $x\in\bbR$),}    
\lb{2.24a}
\end{equation}
denotes the Wronskian of the $C^1$-functions $f$ and $g$. In particular, 
the Floquet solutions $\psi_{\pm} (\zeta, \cdot)$ satisfy the $\zeta$-dependent boundary conditions (cf.\ also \eqref{1.22}), 
\begin{equation}
g(\pi) = \rho_{\pm} (\zeta) g(0), \quad g'(\pi) = \rho_{\pm} (\zeta) g'(0).  \lb{2.25} 
\end{equation}
Moreover, the Floquet solutions $\psi_{\pm} (\zeta, \cdot)$ become 
periodic solutions of \eqref{2.0} if $\zeta^2 = \lambda_{2k}^{\pm}$, $k\in\bbN$, 
that is, 
\begin{align}
\begin{split} 
& \psi_+\big((\lambda_{2k}^{\pm})^{1/2},x\big) = 
\psi_-\big((\lambda_{2k}^{\pm})^{1/2},x\big), \; x \in \bbR, \\
& \quad \text{if } \, \rho_{\pm} \big((\lambda_{2k}^{\pm})^{1/2}\big) = 1 = 
u_+\big((\lambda_{2k}^{\pm})^{1/2}\big) \, \text{ for some } \, k \in\bbN.    
\lb{2.25a}
\end{split} 
\end{align}

Assuming that the square root in \eqref{2.6} is chosen
such that
\begin{equation}
|\rho_+(\zeta)|<1, \quad |\rho_-(\zeta)|>1, \quad
\zeta\in\bbC, \; u_+ (\zeta) \notin [-1,1],  \lb{2.26}
\end{equation}
then
\begin{equation}
\psi_\pm(\zeta,\cdot) \in L^2([x_0,\pm\infty); dx) \, \text{ for all } \,
x_0\in\bbR, \;
\zeta\in\bbC\backslash \big\{\mu_j^{1/2}\big\}_{j\in\bbN}), \; 
u_+ (\zeta) \notin [-1,1].      \lb{2.27}
\end{equation}

Next, we will take a closer look at eigenfunctions and generalized eigenfunctions 
(i.e., root vectors) of $H^P$. For this purpose we suppose that 
$\psi_j(\zeta_j, \cdot)$, $j=1,2$, are distributional solutions of 
\begin{equation}
L \psi_j (\zeta_j,x) = \zeta_j^2 \psi_j (\zeta_j,x), \quad  
\zeta_j \in\bbC, \, j=1,2, \; x\in[0,\pi].    \lb{2.30} 
\end{equation}
Then the fact that 
\begin{equation}
\f{d}{dx} W(\psi_1(\zeta_1,\cdot), \psi_2(\zeta_2,\cdot))(x) = 
\big(\zeta_1^2 - \zeta_2^2\big) \psi_1(\zeta_1,x) \psi_2(\zeta_2,x),  \lb{2.31}
\end{equation}
yields for $x_j \in [0,\pi]$, $j=1,2$, 
\begin{align}
\begin{split} 
& \int_{x_1}^{x_2} dx \, \psi_1(\zeta_1,x) \psi_2(\zeta_2,x)    \\
& \quad = 
\f{W(\psi_1(\zeta_1,\cdot), \psi_2(\zeta_2,\cdot))(x_2) - 
W(\psi_1(\zeta_1,\cdot), \psi_2(\zeta_2,\cdot))(x_1)}{\zeta_1^2 - \zeta_2^2}, 
\quad \zeta_1^2 \neq \zeta_2^2.     \lb{2.32}
\end{split} 
\end{align}
Letting $\zeta_2 \to \zeta_1$ this implies
\begin{align}
& \int_{x_1}^{x_2} dx \, \psi_1(\zeta_1,x) \psi_2(\zeta_1,x)    \lb{2.33} \\
& \quad = 
- \big[W(\psi_1(\zeta_1,\cdot), \psi_2^{\bullet}(\zeta_1,\cdot))(x_2) - 
W(\psi_1(\zeta_1,\cdot), \psi_2^{\bullet}(\zeta_1,\cdot))(x_1)\big]\big/ [2\zeta_1], 
\quad \zeta_1 \in \bbC.    \no
\end{align}
Applying \eqref{2.33} with $x_1 = 0$ and $x_2 = x$ and $\psi_1(\zeta,\cdot)$ 
and $\psi_2(\zeta,\cdot)$ equal to one of $s(\zeta,\cdot)$ and $c(\zeta, \cdot)$, and employing (cf.\ \eqref{2.1})
\begin{equation}
s^{\bullet}(\zeta,0) = c^{\bullet}(\zeta,0)= s^{\bullet\, \prime}(\zeta,0)
= c^{\bullet\, \prime}(\zeta,0) = 0, \quad \zeta \in \bbC,    \lb{2.34}
\end{equation}
then implies 
\begin{align}
s^{\bullet}(\zeta,x) &= 2 \zeta \int_0^x dy \, 
[c(\zeta,x) s(\zeta,y) - s(\zeta,x) c(\zeta,y)] s(\zeta,y),   \lb{2.35} \\ 
c^{\bullet}(\zeta,x) &= 2 \zeta \int_0^x dy \, 
[c(\zeta,x) s(\zeta,y) - s(\zeta,x) c(\zeta,y)] c(\zeta,y),   \lb{2.36} \\ 
s^{\bullet\,\prime}(\zeta,x) &= 2 \zeta \int_0^x dy \, 
[c'(\zeta,x) s(\zeta,y) - s'(\zeta,x) c(\zeta,y)] s(\zeta,y),   \lb{2.37} \\ 
c^{\bullet\,\prime}(\zeta,x) &= 2 \zeta \int_0^x dy \, 
[c'(\zeta,x) s(\zeta,y) - s'(\zeta,x) c(\zeta,y)] c(\zeta,y)    \lb{2.38} 
\end{align}
for $\zeta \in \bbC$, $x \in [0,\pi]$. In particular, choosing $x=\pi$ in 
\eqref{2.36} and \eqref{2.37} implies 
\begin{align}
u_+^{\bullet}(\zeta) &= [c^{\bullet}(\zeta,\pi) + s^{\bullet\,\prime}(\zeta,\pi)]/2   
\no \\
& = \zeta \int_0^{\pi} dy 
\Big[- s(\zeta,\pi) c(\zeta,y)^2 + [c(\zeta,\pi) - s'(\zeta,\pi)] c(\zeta,y) s(\zeta,y)  
\no \\
& \hspace*{1.85cm}   + c'(\zeta,\pi) s(\zeta,y)^2\Big]    \lb{2.39} \\
&= - \zeta s(\zeta,\pi) \int_{0}^{\pi}
dx \, \psi_+(\zeta,x)\psi_-(\zeta,x),  \quad \zeta\in\bbC.     \lb{2.40}
\end{align}
We also refer to \cite[Sect.\ 8.3]{CL85}, \cite[Ch.\ 2]{Ea73}, \cite[Sect.\ 10.8]{In56}, 
\cite[Sect.\ 21.3]{Ti58} in connection with \eqref{2.34}--\eqref{2.39}.

Turning to eigenvectors of $H^P$ associated with some 
$\zeta^2 \in \sigma\big(H^P\big)$, the corresponding eigenvector ansatz 
\begin{equation}
f(\zeta,x) = A c(\zeta,x) + B s(\zeta,x), \quad x \in [0,\pi],    \lb{2.41}
\end{equation}
and the periodic boundary conditions in \eqref{1.3} yield the characteristic 
equation
\begin{align}
\det\left(I_2 - \cM(\zeta)\right) &=
\det\left(\begin{pmatrix} 1-c(\zeta,\pi) & - s(\zeta,\pi) \\
- c'(\zeta,\pi) & 1 - s'(\zeta,\pi) \end{pmatrix}\right)   \no \\ 
&= 2 [1 - c(\zeta,\pi) - s'(\zeta,\pi)]     \no \\
&= 2 [1 - u_+(\zeta)] = 0,     \lb{2.42}
\end{align}
with $I_2$ the identity matrix in $\bbC^2$. 

To discuss the connection between the algebraic multiplicity of an eigenvalue 
of $H^P$ and the order of a zero in equation \eqref{2.42}, we recall the 
following results: 
\begin{align}
& {\det}_{L^2([0,\pi]; dx)} \Big(I - \big(\zeta^2 - \zeta_0^2\big)
\big(H^P - \zeta_0^2 I\big)^{-1}\Big) 
= \f{1 - u_+(\zeta)}{1 - u_+(\zeta_0)}, \quad \zeta \in \bbC, \; 
\zeta_0^2 \in \rho\big(H^P\big),   \lb{2.43} \\
& {\det}_{L^2([0,\pi]; dx)} \Big(I - \big(\zeta^2 - \zeta_0^2\big)
\big(H^D - \zeta_0^2 I\big)^{-1}\Big) 
= \f{s(\zeta,\pi)}{s(\zeta_0,\pi)}, \quad \zeta \in \bbC, \; 
\zeta_0^2 \in \rho\big(H^D\big),   \lb{2.44} \\
& {\det}_{L^2([0,\pi]; dx)} \Big(I - \big(\zeta^2 - \zeta_0^2\big)
\big(H^N - \zeta_0^2 I\big)^{-1}\Big) 
= \f{c'(\zeta,\pi)}{c'(\zeta_0,\pi)}, \quad \zeta \in \bbC, \; 
\zeta_0^2 \in \rho\big(H^N\big).    \lb{2.45}
\end{align}
Here ${\det}_{L^2([0,\pi]; dx)} (\cdot)$ denotes the Fredholm determinant 
of trace class perturbations of the identity operator $I$ in $L^2([0,\pi]; dx)$ 
and we used the fact that 
\begin{align}
& \big(H^P - \zeta_0^2 I\big)^{-1}, \, \big(H^D - \zeta_0^2 I\big)^{-1}, \, 
\big(H^N - \zeta_0^2 I\big)^{-1} 
\in \cB_1\big(L^2([0,\pi]; dx)\big)    \lb{2.46} 
\end{align}
for $\zeta_0^2$ in the corresponding resolvent set. Equations 
\eqref{2.43}--\eqref{2.45} have been derived, for instance, in \cite{GW95} (cf.\ also 
\cite{GM03}). The relevance of \eqref{2.43}--\eqref{2.45} now becomes clear when combined 
with the fact \eqref{1.17b}: 
\begin{align}
& \text{The multiplicity of a zero $\zeta_0$ of $[u_+(\cdot) - 1]$ equals 
the algebraic multiplicity}   \no \\
& \quad \text{of the eigenvalue $\zeta_0^2$ of 
$\sigma\big(H^P\big)$},    \lb{2.47} \\
& \text{the multiplicity of a zero $\zeta_0$ of $s(\cdot,\pi)$ equals 
the algebraic multiplicity}    \no \\
& \quad \text{of the eigenvalue $\zeta_0^2$ of 
$\sigma\big(H^D\big)$},    \lb{2.48} \\
& \text{the multiplicity of a zero $\zeta_0$ of $c'(\cdot,\pi)$ equals 
the algebraic multiplicity}    \no \\ 
& \quad \text{of the eigenvalue $\zeta_0^2$ of 
$\sigma\big(H^N\big)$}.     \lb{2.49} 
\end{align}
Alternatively, the conclusions in \eqref{2.47}--\eqref{2.49} also follow from 
the results in \cite[\&\,2.3]{Na67} (see also \cite[App.\ IV, Sect.\ 2]{Le80}).

To reduce the number of case distinctions necessary in connection with a  
discussion of root vectors of $H^P$ and the multiplicity of a zero $\zeta_0$ in 
\eqref{2.42}, we now assume that $|\zeta_0|$ is sufficiently large so that by the 
asymptotic behavior \eqref{2.10}--\eqref{2.14} and an application of 
Rouch\'e's theorem, $[1 - u_+(\cdot)] = 0$ has at most a zero of multiplicity 
two at $\zeta_0$, and analogously, $s(\cdot,\pi)$ and $c'(\cdot,\pi)$ both have 
at most a simple zero at $\zeta_0$. Combining this with the facts in 
\eqref{2.47}--\eqref{2.49}, this implies that 
\begin{align}
& \text{the algebraic multiplicity of any eigenvalue of $H^P$ outside 
$D(0;R)$}   \no \\
& \quad \text{is at most two for $R>0$ sufficiently large,}    \lb{2.50} \\
& \text{the algebraic multiplicity of any eigenvalue of $H^D$ and $H^N$ 
outside $D(0;R)$}    \no \\
& \quad \text{equals one for $R>0$ sufficiently large,}    \lb{2.51} 
\end{align}
where $D(0;R) \subset \bbC$ denotes the open disk in $\bbC$ with center 
at the origin and radius $R>0$. In particular, for $R>0$ sufficiently large, all eigenvalues of $H^D$ and $H^N$ outside $D(0;R)$ are thus simple. 

To describe the root vectors of $H^P$ we now consider the following cases:

\medskip 

$\textbf{Case\,(I):}$ Suppose that $\zeta_0$ is a simple zero of $[1 - u_+(\cdot)]$,  
that is, 
\begin{equation}
u_+(\zeta_0) = 1, \quad u_+^{\bullet} (\zeta_0) \neq 0.    \lb{2.51a}
\end{equation}

\smallskip 

Then $\zeta_0^2$ is a simple eigenvalue of $H^P$ with corresponding 
eigenspace 
\begin{equation}
\ker\big(H^P - \zeta_0^2 I\big) = \begin{cases} \text{lin.\,span} \, 
\big\{s(\zeta_0,\pi) c(\zeta_0,\cdot) + [1 - c(\zeta_0,\pi)] s(\zeta_0,\cdot)\big\}, \\
\hspace*{2.8cm}  \text{if $s(\zeta_0,\pi) \neq 0$ or $c(\zeta_0,\pi) \neq 1$}, \\
\text{lin.\,span} \, \big\{[1 - s'(\zeta_0,\pi)] c(\zeta_0,\cdot) + c'(\zeta_0,\pi)] s(\zeta_0,\cdot)\big\}, \\
\hspace*{2.9cm}  \text{if $c'(\zeta_0,\pi) \neq 0$ or $s'(\zeta_0,\pi) \neq 1$.}
\end{cases}     \lb{2.52}
\end{equation}
One observes that 
\begin{equation}
s(\zeta_0,\pi) = c'(\zeta_0,\pi) = 0 \, \text{ and } \, 
c(\zeta_0,\pi) = s'(\zeta_0,\pi) = 1     \lb{2.53}
\end{equation}
contradicts the assumption that $\zeta_0$ is a simple zero of $[1 - u_+(\cdot)]$. 

\medskip 

$\textbf{Case\,(II):}$ Suppose that $\zeta_0$ is a double zero of 
$[1 - u_+(\cdot)]$, that is,
\begin{equation}
u_+(\zeta_0) = 1, \quad u_+^{\bullet} (\zeta_0) = 0, 
\quad u_+^{\bullet \, \bullet} (\zeta_0) \neq 0,    \lb{2.53a}
\end{equation}
and 
\begin{equation}
s(\zeta_0,\pi) \neq 0, \quad c'(\zeta_0,\pi) = 0.    \lb{2.53b}
\end{equation} 

\smallskip 

Then $W(c(\zeta_0,\cdot),s(\zeta_0,\cdot))(\pi) =1$ and $u_+(\zeta_0) = 1$ yield 
$c(\zeta_0,\pi) = s'(\zeta_0,\pi) = 1$ and hence \eqref{2.39} implies 
\begin{equation}
\int_0^{\pi} dy \, c(\zeta_0,y)^2 = 0    \lb{2.54}
\end{equation}
and thus \eqref{2.38} yields 
\begin{equation}
c^{\bullet \, \prime} (\zeta_0, \pi) = 0.     \lb{2.55}
\end{equation}
However, \eqref{2.55} would imply the contradiction that $\zeta_0^2$ is an 
eigenvalue of $H^N$ of algebraic multiplicity equal to two by \eqref{2.38} and \eqref{2.45}. Hence this case cannot occur for $|\zeta_0| > R$ by hypothesis.  

\medskip 

$\textbf{Case\,(III):}$ Suppose that $\zeta_0$ is a double zero of 
$[1 - u_+(\cdot)]$, that is, 
\begin{equation}
u_+(\zeta_0) = 1, \quad u_+^{\bullet} (\zeta_0) = 0, 
\quad u_+^{\bullet \, \bullet} (\zeta_0) \neq 0,    \lb{2.55a}
\end{equation}
and 
\begin{equation}
s(\zeta_0,\pi) = 0, \quad c'(\zeta_0,\pi) \neq 0.
\end{equation}

\smallskip 

Then again $W(c(\zeta_0,\cdot),s(\zeta_0,\cdot))(\pi) =1$ and 
$u_+(\zeta_0) = 1$ yield $c(\zeta_0,\pi) = s'(\zeta_0,\pi) = 1$ and hence 
\eqref{2.39} now implies 
\begin{equation}
\int_0^{\pi} dy \, s(\zeta_0,y)^2 = 0    \lb{2.56}
\end{equation}
and thus \eqref{2.35} yields 
\begin{equation}
s^{\bullet} (\zeta_0, \pi) = 0.     \lb{2.57}
\end{equation}
However, \eqref{2.57} would imply the contradiction that $\zeta_0^2$ is an 
eigenvalue of $H^D$ of  algebraic multiplicity equal to two by \eqref{2.36} and \eqref{2.45}. Hence also this case cannot occur for $|\zeta_0| > R$ by 
hypothesis.  

\medskip 

$\textbf{Case\,(IV):}$ Suppose that $\zeta_0$ is a double zero of 
$[1 - u_+(\cdot)]$, that is, 
\begin{equation}
u_+(\zeta_0) = 1, \quad u_+^{\bullet} (\zeta_0) = 0, 
\quad u_+^{\bullet \, \bullet} (\zeta_0) \neq 0,    \lb{2.57a}
\end{equation}
and 
\begin{equation}
s(\zeta_0,\pi) \neq 0, \quad c'(\zeta_0,\pi) \neq 0.
\end{equation} 

\smallskip 

Then $\zeta_0^2$ is an eigenvalue of $H^P$ of algebraic multiplicity equal to 
two with corresponding eigenspace 
\begin{equation}
\ker\big(H^P - \zeta_0^2 I\big) = \text{lin.\,span} \, 
\big\{s(\zeta_0,\pi) c(\zeta_0,\cdot) + [1 - c(\zeta_0,\pi)] s(\zeta_0,\cdot)\big\} 
\lb{2.58}
\end{equation}
and generalized eigenspace 
\begin{align}
& \ker\Big(\big(H^P - \zeta_0^2 I\big)^2\Big) \Big\backslash 
\ker\big(H^P - \zeta_0^2 I\big) \no \\
& \quad = \text{lin.\,span} \, \bigg\{
\Big\{\big[c^{\bullet}(\zeta_0,\pi)/[1 - c(\zeta_0,\pi)]\big] +
[s^{\bullet}(\zeta_0,\pi)/s(\zeta_0,\pi)]\Big\} c(\zeta_0,\cdot) \no \\
& \hspace*{2.4cm} + c^{\bullet}(\zeta_0,\cdot) 
+ \big[[1 - c(\zeta_0,\pi)]/s(\zeta_0,\pi)\big] s^{\bullet}(\zeta_0,\cdot)\bigg\}.  
\lb{2.59} 
\end{align}

The fact \eqref{2.59} is obtained as follows: First, one makes the general 
ansatz for the root vector $g(\zeta_0,\cdot)$ of $H^P$ associated with 
$\zeta_0^2$,  
\begin{equation}
g(\zeta_0,x) = A c(\zeta_0,x) + B s(\zeta_0,x) + C c^{\bullet}(\zeta_0,x) 
+ D s^{\bullet}(\zeta,x), \quad x \in [0,\pi],   \lb{2.60} 
\end{equation}
and requires that 
\begin{equation}
g(\zeta_0,\cdot) \in \dom\big(H^P\big).    \lb{2.61}
\end{equation}
Because of 
\eqref{2.58} one can, without loss of generality, either set $A=0$ or $B=0$ 
in \eqref{2.60}. Choosing 
\begin{equation} 
B=0     \lb{2.62}
\end{equation} 
in \eqref{2.60}, and noting that 
\begin{equation}
\big(L - \zeta_0^2\big) \f{1}{2 \zeta_0} 
\begin{pmatrix} c^{\bullet}(\zeta_0,\cdot) \\ s^{\bullet}(\zeta_0,\cdot)
\end{pmatrix} = \begin{pmatrix} c(\zeta_0,\cdot) \\ s(\zeta_0,\cdot)
\end{pmatrix}    \lb{2.63}
\end{equation}
in the sense of distributions, the requirement 
\begin{equation}
\big(L - \zeta_0^2\big) g(\zeta,\cdot) = C c(\zeta_0,\cdot) 
+ D s(\zeta,\cdot) \in \dom\big(H^P\big)   \lb{2.64}
\end{equation}
then yields 
\begin{equation}
[1 - c(\zeta_0,\pi)] C - s(\zeta_0,\pi) D = 0,    \lb{2.65}
\end{equation}
and solving for $A$ in connection with \eqref{2.61} implies \eqref{2.59}. 

As a consequence, all eigenvalues $\lambda_{2k}^+$ of $H^P$ of 
sufficiently large magnitude, necessarily either belong to $\textbf{Case\,(I)}$ 
or to $\textbf{Case\,(IV)}$. 

For additional discussions of root vectors we also refer, for instance, 
to \cite[App.\ IV.2]{Le96}, \cite[Sect.\ 1.3]{Ma86} and \cite[\&\,I.2]{Na67}. 

In connection with the concept of biorthogonality we recall the 
elementary fact that if $H^P f_k = \lambda_k f_k$, $f_k \in \dom\big(H^P\big)$ 
and $\big(H^P\big)^* g_{\ell} = \ol{\lambda_{\ell}} g_{\ell}$, 
$g_{\ell} \in \dom\big(\big(H^P\big)^*\big)$, with 
$\lambda_k \neq \lambda_{\ell}$, then
\begin{align}
\lambda_k (g_{\ell}, f_k)_{\cH} = \big(g_{\ell}, H^P f_k\big)_{\cH} 
= \big(\big(H^P\big)^* g_{\ell}, f_k\big)_{\cH} 
= \lambda_{\ell} (g_{\ell}, f_k)_{\cH}     \lb{2.79}
\end{align}
yields $[\lambda_k - \lambda_{\ell}] (g_{\ell}, f_k)_{\cH} = 0$ and hence 
\begin{equation}
(g_{\ell}, f_k)_{\cH} = 0 \, \text{ since $\lambda_k \neq \lambda_{\ell}$.} 
\lb{2.80}
\end{equation}
In particular, since the boundary conditions in $H^P$ are self-adjoint, this 
shows that $g_{\ell}(x) = \ol{f_{\ell}(x)}$, $x \in [0,\pi]$, satisfies 
$g_{\ell} \in \dom\big(\big(H^P\big)^*\big)$ and 
$\big(H^P\big)^* g_{\ell} = \ol{\lambda_{\ell}} g_{\ell}$, whenever $f_{\ell}$ 
satisfies $H^P f_{\ell} = \lambda_{\ell} f_{\ell}$, $f_{\ell} \in \dom\big(H^P\big)$.  

Of course, completely analogous considerations apply to $H^{AP}$. 

Next, we note that the entire functions $u_+(\cdot) - \cos(t)$, $t\in [0, 2\pi)$, 
$s(\cdot,\pi)$, and $c'(\cdot,\pi)$ are all nonconstant, in particular, the characteristic functions associated with the boundary value problems 
$Lg = \zeta^2 g$, $g, g' \in AC([0,\pi])$, and the corresponding boundary conditions in $H(t)$, $t \in [0,2\pi)$, $H^D$, and $H^N$ do not vanish 
identically. By definition, the latter property characterizes the boundary 
conditions in $H(t)$, $t \in [0,2\pi)$, $H^D$, and $H^N$ as {\it nondegenerate}. This yields the following special case of a result proved, for instance, in 
\cite[Theorem\ 1.3.1]{Ma86}: 

\begin{theorem} \lb{t2.1}
Assume $V \in L^2([0,\pi]; dx)$, then the system of root vectors of 
$H(t)$, $t \in [0,2\pi)$, $H^D$, and $H^N$, respectively, is complete  
in $L^2([0,\pi]; dx)$.
\end{theorem} 

It should be noted that since the boundary conditions for $H(t)$ in 
\eqref{1.22} are regular in the sense of Birkhoff (cf.\ \cite[\&\,4.8]{Na67}, and for 
various generalizations, see, e.g., \cite{EF78}, \cite{EFZ05}, \cite{Fr84}, \cite{FR95}, 
\cite{Wa77}), every vector in $L^2([0,\pi]; dx)$ can be represented as a bracketed 
series of root vectors which also implies the completeness of the system of root 
vectors of $H(t)$, $t \in [0,2\pi)$. 

In conclusion, we also recall the resolvent formula for the operator $H(t)$, 
$t \in [0,2\pi)$; the special periodic case, $t=0$, will play a particular role 
in Section \ref{s3}. 
\begin{align}
\begin{split}
\big(\big(H(t) - \zeta^2 I\big)^{-1}f\big)(x) 
= \int_0^{\pi} dy \, G_{H(t)}\big(\zeta^2,x,y\big) f(y),&       \lb{2.81} \\
\zeta^2 \in \rho(H(t)), \; t \in [0, 2\pi), \; f \in L^2([0,\pi]; dx),&    
\end{split} 
\end{align} 
where (cf., e.g., \cite[\&\,3.7]{Na67}, \cite[Sect.\ III.16]{RS78})
\begin{align} 
& G_{H(t)}\big(\zeta^2,x,y\big) = \f{1}{W(\psi_+(\zeta),\psi_-(\zeta))} 
\begin{cases} 
\psi_-(\zeta,x) \psi_+(\zeta,y), & x \leq y, \\
\psi_+(\zeta,x) \psi_-(\zeta,y), & x \geq y, \end{cases}   \lb{2.82} \\
& \qquad + \f{\psi_+(\zeta,x) \psi_-(\zeta,y)}{[e^{it} \rho_-(\zeta) - 1] 
W(\psi_+(\zeta),\psi_+(\zeta))}  
+ \f{\psi_-(\zeta,x) \psi_+(\zeta,y)}{[e^{-it} \rho_-(\zeta) - 1]    
W(\psi_+(\zeta),\psi_-(\zeta))}      \no \\ 
& \quad =  \f{s(\zeta,\pi)}{2[u_+(\zeta) - \cos(t)]} c(\zeta,x) c(\zeta,y)
- \f{c'(\zeta,\pi)}{2[u_+(\zeta) - \cos(t)]} s(\zeta,x) s(\zeta,y)    \no \\
& \qquad + \f{\begin{cases} e^{-i t} - c(\zeta,\pi), & x \leq y, \\
- e^{it} + s'(\zeta,\pi), & x \geq y, \end{cases}}
{2[u_+(\zeta) - \cos(t)]} c(\zeta,x) s(\zeta,y)      \lb{2.83}  \\
& \qquad + \f{\begin{cases} - e^{-i t} + s'(\zeta,\pi), & x \leq y, \\
e^{it} - c(\zeta,\pi), & x \geq y, \end{cases}}
{2[u_+(\zeta) - \cos(t)]} s(\zeta,x) c(\zeta,y),    \no \\
& \hspace*{1.6cm} \zeta^2 \in \rho(H(t)), \; t \in [0, 2\pi), \; x,y \in [0,\pi].    \no 
\end{align}

\section{More Background Properties on $H^P$ and $H^D$} \lb{s3}

In this section we take a closer look at root vectors of $H^P$ and $H^D$.

We start by partitioning the spectrum of $H^D$ appropriately: First, let 
$k_0 \in \bbN$ be sufficiently large such that every disk 
$D_{2k} = \{ z\in\bbC \,|\, |z - 4 k^2| \leq 2 |c| + 1\}$, $k \geq k_0$, 
contains precisely two points of $\sigma \big(H^P\big)$ and one point of 
$\sigma\big(H^D\big)$. It follows from \eqref{2.15}--\eqref{2.17} that the disk 
$D_0 = \{ z\in\bbC \,|\, |z| \leq 4 k_0^2 + 2 |c| + 1\}$ contains precisely 
$2 k_0 + 1$ points of $\sigma \big(H^P\big)$ and $k_0$ points of 
$\sigma \big(H^D\big)$. 

Now we split up $\bbN_0$ into three disjoint subsets as follows,
\begin{equation}
\bbN_0 = \big\{\bbN_{k_0} \cup\{0\}\big\} \cup \bbN_s \cup \bbN_m,    \lb{3.1}
\end{equation} 
where $N_{k_0} = \{1, \dots , k_0\}$, $\bbN_s$ is the subset of all 
positive integers $k$ such that the points of $\sigma \big(H^P\big)$ in 
the disk $D_{2k}$ are simple and distinct, and $\bbN_m$ is the subset of all positive integers $k$ such that there is only one point of 
$\sigma \big(H^P\big)$ inside the disk $D_{2k}$ of algebraic multiplicity 
equal to two. 

For the $2 k_0 + 1$ eigenvalues of $\sigma \big(H^P\big)$ inside the disk 
$D_0$ we fix an arbitrary labeling 
$\lambda_0, \lambda_2^{\pm}, \dots , \lambda_{2 k_0}^{\pm}$ and two 
associated biorthogonal root systems 
\begin{equation}
\{\phi_0, \phi_k^{\pm}\}_{k\in\bbN_{k_0}}, \, \text{ and } \, 
\{\chi_0, \chi_k^{\pm}\}_{k\in\bbN_{k_0}}      \lb{3.2}
\end{equation}
of $H^P$ and $\big(H^P\big)^*$, respectively. 

Moreover, the set $\bbN_s$ can be further disjointly decomposed into 
$\bbN_s = \bbN_s' \cup \bbN_s''$, where the distinct points 
$\lambda_{2k}^+$ and $\lambda_{2k}^-$ of 
$\sigma \big(H^P\big)$ in the disk $D_{2k}$, for $k \in \bbN_s'$ differ from 
$\mu_{2k} \in \sigma\big(H^D\big)$, while for $k \in \bbN_s''$ one of 
the distinct points $\lambda_{2k}^+$ and $\lambda_{2k}^-$ of 
$\sigma \big(H^P\big)$ in the disk $D_{2k}$ coincides with 
$\mu_{2k} \in \sigma\big(H^D\big)$. 

For $k \in\bbN_s'$ we fix an arbitrary labeling of $\lambda_{2k}^+$ and $\lambda_{2k}^-$, 
and set $\lambda_{2k}^+ = \mu_{2k} \neq \lambda_{2k}^-$ for $k \in \bbN_s''$. 

Finally, for $k \in \bbN_m$ one has $\lambda_{2k}^+ = \lambda_{2k}^-$ and we further  
disjointly decompose $\bbN_m = \bbN_m' \cup \bbN_m''$, where 
\begin{equation}
\bbN_m' =\{k \in\bbN_m \,|\, \lambda_{2k}^+ = \lambda_{2k}^- = \mu_{2k}\}, \quad 
\bbN_m'' =\{k \in\bbN_m \,|\, \lambda_{2k}^+ = \lambda_{2k}^- \neq \mu_{2k}\}. 
\lb{3.3}
\end{equation}

In the following we will use the connection 
$z = \zeta^2 \in \bbC$ and predominantly use the variable $\zeta \in \bbC$. 
In addition, we also agree to use the notation 
\begin{equation}
\zeta_{k} = \big[\mu_{k}\big]^{1/2}, \quad 
\xi_0 = \big[\lambda_0\big]^{1/2}, \;  
\xi_{k}^{\pm} = \big[\lambda_{k}^{\pm}\big]^{1/2}, \quad 
\omega_{0} = \big[\kappa_{0}\big]^{1/2}, \; 
\omega_{k} = \big[\kappa_{k}\big]^{1/2}, \; k \in \bbN.    \lb{3.4}
\end{equation} 
In particular, we use the enumeration where (cf.\ \eqref{2.15}, \eqref{2.16}) 
for $k \geq k_0$, with $k_0 \in \bbN$ sufficiently large, 
\begin{equation}
\zeta_{2k}^{\pm} \underset{k \to \infty}{=}  2k + \Oh\big(k^{-1}\big), \quad 
\xi_{2k}^{\pm} \underset{k \to \infty}{=} 2k + \Oh\big(k^{-1}\big),   \lb{3.4a}
\end{equation}
and choose some enumeration of the $2k_0 +1$ points 
$\xi_0, \xi_2^{\pm}, \dots , \xi_{2 k_0}^{\pm}$
and the $k_0$ points $\zeta_2, \dots, \zeta_{2 k_0}$ (counting multiplicity) inside the disk $\big\{\zeta \in\bbC \,|\, |\zeta| \leq [4 k_0^2 + 2 |c| + 1]^{1/2}\big\}$.

Moreover, in the proof of Lemma \ref{l4.2} we need to enumerate all zeros 
of $u_+(\cdot) - 1$, $u_+^{\bullet}(\cdot)$, and $s(\cdot,\pi)$. Due to our 
choice of $\lambda_0 = [\xi_0]^2, \lambda_k^{\pm} = [\xi_k^{\pm}]^2$, 
$\mu_k = [\zeta_k]^2$, $\kappa_0 = [\omega_0]^2$, 
$\kappa_k = [\omega_k]^2$, $k\in\bbN$, we then use the additional 
enumeration 
\begin{equation}
\zeta_{-k} = - \zeta_{k}, \quad 
\pm\xi_0, \; \xi_{-k}^{\pm} = - \xi_{k}^{\pm},  \quad 
\pm\omega_0, \; \omega_{-k} = - \omega_{k}, \quad k \in \bbN.    \lb{3.4b}
\end{equation} 

Summarizing our notation concerning $\bbN_s$ and $\bbN_m$, one then 
has the following scenarios:
\begin{align}
& u_+(\xi_{2k}^{\pm}) =1, \quad u_+^{\bullet}(\xi_{2k}^{\pm}) \neq 0, \quad 
\xi_{2k}^+ \neq \xi_{2k}^-, \quad k \in \bbN_s,    \lb{3.4c} \\
& s(\xi_{2k}^{\pm},\pi) \neq 0, \quad k \in \bbN_s^{\prime},    \lb{3.4d} \\
& s(\xi_{2k}^+,\pi) = 0, \quad s^{\bullet}(\xi_{2k}^+,\pi) \neq 0,
\quad s(\xi_{2k}^-,\pi) \neq 0, \quad 
 k \in \bbN_s^{\prime \prime},    \lb{3.4e} \\
 & u_+(\xi_{2k}^{\pm}) =1, \quad u_+^{\bullet}(\xi_{2k}^{\pm}) = 0,
 \quad u_+^{\bullet\,\bullet}(\xi_{2k}^{\pm}) \neq 0, \quad 
 \xi_{2k}^+ = \xi_{2k}^-, \quad k \in \bbN_m,    \lb{3.4f} \\
& s(\xi_{2k}^{\pm},\pi) = 0, \quad s^{\bullet}(\xi_{2k}^{\pm},\pi) \neq 0,
\quad k \in \bbN_m^{\prime},    \lb{3.4g} \\ 
& s(\xi_{2k}^{\pm},\pi) \neq 0, \quad k \in \bbN_m^{\prime\prime}.    \lb{3.4h} 
\end{align}

To describe root vectors of $H^P$ and $\big(H^P\big)^*$ we will employ the resolvent $R^P(\cdot)$ of $H^P$. The latter corresponds to the case $t=0$ in 
\eqref{2.83} and hence is of the form,
\begin{align}
& \Big(R^P\big(\zeta^2\big) f\Big)(x) 
= \Big(\big(H^P - \zeta^2 I\big)^{-1} f\Big)(x)   
= \int_0^{\pi} dy \, G_{H^P} \big(\zeta^2,x,y\big) f(y),    \lb{3.5A} \\ 
& \hspace*{5.55cm} \zeta^2 \in \rho\big(H^P\big), \; f \in L^2([0,\pi]; dx), \no \\
& G_{H^P} \big(\zeta^2,x,y\big) 
= \f{s(\zeta,\pi)}{2[u_+(\zeta) -1]} c(\zeta,x) c(\zeta,y) 
- \f{c'(\zeta,\pi)}{2[u_+(\zeta) -1]} s(\zeta,x) s(\zeta,y)    \no \\
& \hspace*{2.45cm} + \f{A(\zeta)}{2[u_+(\zeta) -1]} c(\zeta,x) s(\zeta,y) 
+ \f{B(\zeta)}{2[u_+(\zeta) -1]} s(\zeta,x) c(\zeta,y)   \no \\
& \hspace*{2.45cm}  + \Omega(\zeta,x,y;f;A,B),  \quad 
\zeta^2 \in \rho\big(H^P\big), \;  x,y \in [0,\pi],    \lb{3.5}
\end{align}
where $A(\zeta)$ and $B(\zeta)$ can be chosen to be either 
$[1 - c(\zeta,\pi)]$ or $[s'(\zeta,\pi) - 1]$, 
\begin{equation}
A(\zeta), \, B(\zeta) \in \big\{[1 - c(\zeta,\pi)], [s'(\zeta,\pi) - 1]\big\},    \lb{3.6}
\end{equation}
and $\Omega(\zeta,x,y;f;A,B)$ is entire with respect to $\zeta$. 

\medskip

$\mathbf{(I)}$ We start with eigenvectors of $H^P$ associated with the 
eigenvalues $\lambda_{2k}^+ = \big[\xi_{2k}^+\big]^2$ for 
$k \in \bbN_s = \bbN_s' \cup \bbN_s''$. 

We first treat the case $k \in \bbN_s'$, where $\xi_{2k}^+ \neq \xi_{2k}^-$, 
$\zeta_{2k} \neq \xi_{2k}^{\pm}$, and hence 
$s(\xi_{2k}^{\pm},\pi) \neq 0$. Focusing temporarily on $\xi_{2k}^+$, we 
separately consider the cases where $s'(\xi_{2k}^+,\pi) \neq 1$ and 
$s'(\xi_{2k}^+,\pi) = 1$. We start with the case 

\medskip

\noindent $s'(\xi_{2k}^+,\pi) \neq 1$: Then 
\begin{equation}
s'(\xi_{2k}^+,\pi) - 1 = s'(\xi_{2k}^+,\pi) - u_+(\xi_{2k}^+) = - u_-(\xi_{2k}^+) 
= -[c(\xi_{2k}^+,\pi) - 1],     \lb{3.7} 
\end{equation}
and the identity
\begin{equation}
u_+(\zeta)^2 -1 - u_-(\zeta)^2 = c'(\zeta,\pi) s(\zeta,\pi),    \lb{3.7a}
\end{equation}
yields 
\begin{equation}
\f{u_-(\xi_{2k}^+)}{s(\xi_{2k}^+,\pi)^2} = - \f{c'(\xi_{2k}^+,\pi)}{s(\xi_{2k}^+,\pi)}. 
\end{equation}
Since $\lambda_{2k}^+ = \big[\xi_{2k}^+\big]^2$ is a simple pole of 
$R^P(\cdot)$, one obtains from \eqref{3.5}
\begin{equation}
- \Big({\rm Res}_{\zeta^2=\lambda_{2k}^+} R^P\big(\zeta^2\big) f\Big)(x) =
\f{\xi_{2k}^+ s(\xi_{2k}^+,\pi)}{u_+^{\bullet} (\xi_{2k}^+)} \psi_{\pm}(\xi_{2k}^+,x) 
\int_0^{\pi} dy \, \psi_{\pm}(\xi_{2k}^+,y) f(y),     \lb{3.8}
\end{equation}
with 
\begin{equation}
\psi_{\pm}(\xi_k^+,x) = c(\xi_k^+,x) - [u_-(\xi_k^+)/s(\xi_k^+,\pi)] s(\xi_k^+,x).    
\lb{3.9}
\end{equation}

Next, we turn to the case

\medskip

\noindent $s'(\xi_{2k}^+,\pi) = 1$: Then the identity \eqref{3.7a} 
and the fact
\begin{equation}
2 u_+(\xi_{2k}^+) = c(\xi_{2k}^+,\pi) + s'(\xi_{2k}^+,\pi) = 2    \lb{3.11}
\end{equation}
imply
\begin{equation}
c(\xi_{2k}^+,\pi) = 1, \quad u_-(\xi_{2k}^+) = 0,   \, \text{ and hence } \, 
c'(\xi_{2k}^+,\pi) s(\xi_{2k}^+,\pi) =0.    \lb{3.12}
\end{equation}
Since $k \in \bbN_s'$ and thus $s(\xi_{2k}^+,\pi) \neq 0$, one obtains that 
$c'(\xi_{2k}^+,\pi) = 0$ and \eqref{3.5} implies 
\begin{equation}
- \Big({\rm Res}_{\zeta^2=\lambda_{2k}^+} R^P\big(\zeta^2\big) f\Big)(x) =
\f{\xi_{2k}^+ s(\xi_{2k}^+,\pi)}{u_+^{\bullet} (\xi_{2k}^+)} c(\xi_{2k}^+,x) 
\int_0^{\pi} dy \, c(\xi_{2k}^+,y) f(y),     \lb{3.13}
\end{equation}
which coincides with \eqref{3.8} since $u_- (\xi_{2k}^+)=0$ and 
$u_+(\xi_{2k}^+)=1$ now yield 
\begin{equation}
\psi_{\pm}(\xi_{2k}^+,x) = c(\xi_{2k}^+,x).    \lb{3.13a}
\end{equation} 

Analogous arguments apply to $\lambda_{2k}^- = \big[\xi_{2k}^-\big]^2$ 
and hence permit one to introduce 
\begin{align}
\begin{split}
\phi_k^{\pm}(x) &= \bigg[
\f{\xi_{2k}^{\pm} s(\xi_{2k}^{\pm},\pi)}{u_+^{\bullet}(\xi_{2k}^{\pm})}\bigg]^{1/2}
\psi_+ (\xi_{2k}^{\pm},x),    \\
\chi_k^{\pm}(x) &= \ol{\phi_k^{\pm}(x)}, \quad 
x \in [0,\pi], \; k \in \bbN_s'.     \lb{3.14} 
\end{split} 
\end{align}

Employing \eqref{2.40}, one concludes that 
\begin{equation}
\{\phi_k^{\pm}\}_{k\in\bbN_s'} \, \text{ and } \, 
\{\chi_k^{\pm}\}_{k\in\bbN_s'}     \lb{3.15}
\end{equation}
form two biorthogonal systems in $L^2([0,\pi]; dx)$, where 
$\phi_k^{\pm}$, $k\in\bbN_s'$, are eigenvectors of $H^P$, and 
$\chi_k^{\pm}$, $k\in\bbN_s'$, are eigenvectors of $\big(H^P\big)^*$, 
with eigenvalues $\xi_{2k}^{\pm}$ and $\ol{\xi_{2k}^{\pm}}$,
respectively. 

Next, we consider the case $k \in \bbN_s''$ implying that 
$\xi_{2k}^+ = \zeta_{2k} \neq \xi_{2k}^-$. In this case 
\begin{equation}
c(\xi_{2k}^+,\pi) - 1 = s'(\xi_{2k}^+,\pi) - 1 = s(\xi_{2k}^+,\pi) 
= s(\zeta_{2k},\pi) = u_-(\xi_{2k}^+) = 0    \lb{3.16}
\end{equation}
and 
\begin{equation}
c'(\xi_{2k}^+,\pi) = \lim_{\zeta \to \xi_{2k}^+} 
\f{u_+(\zeta)^2 -1 - u_-(\zeta)^2}{s(\zeta,\pi)} 
= \f{2 u_+^{\bullet} (\xi_{2k}^+)}{s^{\bullet} (\xi_{2k}^+,\pi)}.    \lb{3.17}
\end{equation}
Rewriting \eqref{2.40} in the form 
\begin{align}
\begin{split}
& u_+^{\bullet}(\zeta) = \zeta \int_0^{\pi} dy 
\Big[- s(\zeta,\pi) c(\zeta,y)^2 + 2 u_-(\zeta) c(\zeta,y) s(\zeta,y)    \lb{3.18} \\
& \hspace*{2.75cm} + 
\big[u_+(\zeta)^2 - 1 - u_-(\zeta)^2\big] s(\zeta,\pi)^{-1} s(\zeta,y)^2\Big],
\end{split} 
\end{align}
one concludes that 
\begin{equation}
s^{\bullet} (\xi_{2k}^+,\pi) = 2 \xi_{2k}^+ \int_0^{\pi} dy \, s(\xi_{2k}^+,y)^2 
\neq 0.      \lb{3.19}
\end{equation}
Thus, \eqref{3.5} implies 
\begin{equation}
- \Big({\rm Res}_{\zeta^2=\lambda_{2k}^+} R^P\big(\zeta^2\big) f\Big)(x) =
\f{-2 \xi_{2k}^+}{s^{\bullet} (\xi_{2k}^+,\pi)} s(\xi_{2k}^+,x) 
\int_0^{\pi} dy \, s(\xi_{2k}^+,y) f(y).     \lb{3.19a}
\end{equation}

Hence, introducing
\begin{align}
& \phi_k^+(x) = \bigg[\f{-2 \xi_{2k}^+}{s^{\bullet}(\xi_{2k}^+,\pi)}\bigg]^{1/2} 
s(\xi_{2k}^+,x),    \no \\
& \phi_k^-(x) 
= \bigg[\f{\xi_{2k}^- s(\xi_{2k}^-,\pi)}{u_+^{\bullet}(\xi_{2k}^-)}\bigg]^{1/2} 
\psi_+ (\xi_{2k}^-,x),     \lb{3.20} \\
& \chi_k^{\pm} (x) = \ol{\phi_k^{\pm} (x)}, \quad x \in [0,\pi], \; k \in \bbN_s''  
\no 
\end{align}
(the case for $\xi_{2k}^-$ being analogous to that in \eqref{3.14}), one 
obtains that 
\begin{equation}
\{\phi_k^{\pm}\}_{k\in\bbN_s''} \, \text{ and } \, 
\{\chi_k^{\pm}\}_{k\in\bbN_s''}       \lb{3.22}
\end{equation}
represent two biorthogonal systems in $L^2([0,\pi]; dx)$. 
Here $\phi_k^{\pm}$, $k\in\bbN_s''$, are eigenvectors of $H^P$, and 
$\chi_k^{\pm}$, $k\in\bbN_s''$, are eigenvectors of $\big(H^P\big)^*$, 
with eigenvalues $\xi_{2k}^{\pm}$ and $\ol{\xi_{2k}^{\pm}}$,
respectively. 

\medskip

$\mathbf{(II)}$ We continue with root vectors associated with the eigenvalues 
$\lambda_{2k}^+ = \big[\xi_{2k}^+\big]^2$ for 
$k \in \bbN_m = \bbN_m' \cup \bbN_m''$. 

We start with $k \in \bbN_m'$, where $\xi_{2k}^{\pm} = \zeta_{2k}$ and 
\begin{align}
\begin{split} 
c(\zeta_{2k},\pi) - 1 &= s'(\zeta_{2k},\pi) - 1 = s(\zeta_{2k},\pi) 
= c'(\zeta_{2k},\pi)      \lb{3.23} \\
&= u_+(\zeta_{2k}) - 1 = u_+^{\bullet}(\zeta_{2k}) = u_-(\zeta_{2k}) = 0,
\end{split} 
\end{align}
and hence every nonzero solution of \eqref{2.0} with 
$\zeta = \xi_{2k}^{\pm} = \zeta_{2k}$  
is an eigenfunction of $H^P$. Moreover, for $k\in \bbN_m'$, 
$\big[1 - u_+(\zeta)^2\big]^{1/2}$ is a single-valued function of $\zeta$ for 
$\zeta^2 \in D_{2k}$ and hence the Floquet solutions permit the asymptotic 
expansion
\begin{equation}
\psi_{\pm}(\zeta,x) \underset{\zeta \to \zeta_{2k}}{=} c(\zeta_{2k},x) 
- \f{u_-^{\bullet}(\zeta_{2k}) 
\pm [u_+^{\bullet \bullet}(\zeta_{2k})]^{1/2}}{s^{\bullet} (\zeta_{2k},\pi)} 
s(\zeta_{2k},x) + \oh(1).    \lb{3.24}
\end{equation}
Differentiating \eqref{2.40} with respect to $\zeta$, subsequently taking 
$\zeta = \zeta_{2k}$, yields
\begin{equation}
u_+^{\bullet \bullet} (\zeta_{2k}) = - \zeta_{2k} s^{\bullet}(\zeta_{2k},\pi) 
\int_0^{\pi} dy \, \psi_+(\zeta_{2k},y) \psi_-(\zeta_{2k},y).      \lb{3.25}
\end{equation}
Denoting 
\begin{equation}
d_k = \big[- \zeta_{2k} s^{\bullet}(\zeta_{2k},\pi)
/u^{\bullet\,\bullet} (\zeta_{2k}\big]^{1/2},    \lb{3.25a}
\end{equation}
it follows from \eqref{2.10}--\eqref{2.16} that the asymptotic representations 
\begin{equation}
d_k \underset{k\to\infty}{=} \pi^{-1/2}[1+\oh(1)], \quad
\psi_{\pm}(\zeta_{2k},x) \underset{k\to\infty}{=} e^{\pm i \zeta_{2k} 
x}[1+\oh(1)]
\end{equation} 
hold. Moreover, if we set 
\begin{equation}
\phi_k^+(x) = d_k \psi_+(\zeta_{2k},x), \quad 
\chi_k^+ (x) = \ol{d_k \psi_-(\zeta_{2k},x)}, \quad x \in [0,\pi], \; k \in \bbN_m',
\lb{3.26}
\end{equation}
then \eqref{3.25} implies 
\begin{equation}
\big(\phi_k^+,\chi_k^+\big)_{L^2([0,\pi]; dx)} = 1, \quad k \in \bbN_m'. 
\end{equation}
Furthermore, one finds that the numbers 
\begin{align}
\begin{split}
& \alpha_k = \big(\chi_k^+, \ol{\chi_k^+}\big)_{L^2([0,\pi]; dx)}, \quad 
\beta_k = \big(\phi_k^+, \ol{\phi_k^+}\big)_{L^2([0,\pi]; dx)},  \\
& \gamma_k = \big(\ol{\chi_k^+} - \alpha_k \phi_k^+, \ol{\phi_k^+} 
- \beta_k \chi_k^+\big)_{L^2([0,\pi]; dx)}, \quad k \in \bbN_m',
\end{split} 
\end{align} 
satisfy the asymptotic relations 
\begin{equation}
\alpha_k \underset{k\to\infty}{=} \oh(1), \quad 
\beta_k \underset{k\to\infty}{=} \oh(1), \quad 
\gamma_k \underset{k\to\infty}{=} 1+\oh(1),   
\end{equation}
and the fact that
\begin{equation}
\big(\ol{\chi_k^+} - \alpha_k \phi_k^+,\chi_k^+\big)_{L^2([0,\pi]; dx)} =
\big(\ol{\phi_k^+} - \beta_k \chi_k^+,\phi_k^+\big)_{L^2([0,\pi]; dx)} = 0, 
\quad k \in \bbN_m'. 
\end{equation}
At this point we define the functions
\begin{align}
\begin{split}
\phi_k^- (x) = \gamma_k^{-1/2} \big[\ol{\chi_k^+(x)} - \alpha_k \phi_k^+(x)\big], \quad 
\chi_k^- (x) = \gamma_k^{-1/2} \big[\ol{\phi_k^+(x)} - \beta_k \chi_k^+(x)\big],& \\
x \in [0,\pi], \; k \in \bbN_m',&    \lb{3.27}
\end{split}
\end{align} 
and conclude that 
\begin{equation}
\{\phi_k^{\pm}\}_{k\in\bbN_m'} \, \text{ and } \, \{\chi_k^{\pm}\}_{k\in\bbN_m'}
\end{equation}
represent two biorthogonal systems in $L^2([0,\pi]; dx)$. Here $\phi_k^{\pm}$, 
$k\in\bbN_m'$, are eigenvectors of $H^P$, and 
$\chi_k^{\pm}$, $k\in\bbN_m'$, are eigenvectors of $\big(H^P\big)^*$, with 
eigenvalues $\zeta_{2k}$ and $\ol{\zeta_{2k}}$, repectively. 

Finally, we consider the case $k\in\bbN_m''$, where 
$\xi_{2k}^+ = \xi_{2k}^- \neq \zeta_{2k}$. In this case 
$s(\xi_{2k}^{\pm},\pi) \neq 0$, and hence \eqref{2.40} implies
\begin{equation}
\int_0^{\pi} dy \, \psi_+(\xi_{2k}^{\pm},y) \psi_-(\xi_{2k}^{\pm},y) = 0   \lb{3.32}
\end{equation}
(with $\xi_{2k}^+ = \xi_{2k}^-$), and 
\begin{align}
\begin{split}
\psi_+(\xi_{2k}^{\pm},x) = \psi_-(\xi_{2k}^{\pm},x) = c(\xi_{2k}^{\pm},x) 
- \big[u_-(\xi_{2k}^{\pm})/s(\xi_{2k}^{\pm},\pi)\big] s(\xi_{2k}^{\pm},x),&    
\lb{3.33} \\
x \in [0,\pi], \; k \in \bbN_m''.&  
\end{split} 
\end{align}
Differentiating \eqref{2.40} with respect to $\zeta$, subsequently taking 
$\zeta = \xi_{2k}^{\pm}$, yields
\begin{equation} 
- u_+^{\bullet \bullet} (\xi_{2k}^{\pm}) = \xi_{2k}^{\pm} s(\xi_{2k}^{\pm},\pi) 
\int_0^{\pi} dy \, \psi_+(\xi_{2k}^{\pm},y) \big[\psi_+^{\bullet} (\xi_{2k}^{\pm},y) 
+ \psi_-^{\bullet} (\xi_{2k}^+,y)\big].    \lb{3.34}
\end{equation}
Employing the identity
\begin{equation}
\psi_+(\zeta,x) + \psi_-(\zeta,x) = 2 c(\zeta,x) - 2 [u_-(\zeta)/s(\zeta,\pi)] 
s (\zeta,x)),    \lb{3.34a}
\end{equation}
to compute $\psi_+^{\bullet} (\xi_{2k}^{\pm},\cdot) 
+ \psi_-^{\bullet} (\xi_{2k}^{\pm},\cdot)$, the functions 
\begin{align}
\phi_k^+(x) &= \f{-1}{u_+^{\bullet \bullet} (\xi_{2k}^{\pm})} \psi_+(\xi_{2k}^{\pm},x), 
\no \\
\phi_k^-(x) &= \xi_{2k}^{\pm} s(\xi_{2k}^{\pm},\pi) 
\big[\psi_+^{\bullet}(\xi_{2k}^{\pm}, x) + \psi_-^{\bullet}(\xi_{2k}^{\pm}, x)\big]   
\no \\
& = 2 \xi_{2k}^{\pm} \big[s(\xi_{2k}^{\pm},\pi) c^{\bullet}(\xi_{2k}^{\pm},x) 
- u_-(\xi_{2k}^{\pm}) s^{\bullet}(\xi_{2k}^{\pm},x)\big]   \no \\
& \quad + 2 \big[[s^{\bullet}(\xi_{2k}^{\pm},\pi)/s(\xi_{2k}^{\pm},\pi)] 
u_-(\xi_{2k}^{\pm})  
- u_-^{\bullet}(\xi_{2k}^{\pm})\big] \xi_{2k}^{\pm} s(\xi_{2k}^{\pm},x), 
\lb{3.35} \\ 
\chi_k^+(x) &= \overline{\phi_k^+(x)},    \no \\
\chi_k^-(x) &= \overline{\phi_k^-(x)} - \big(\phi_k^-, \ol{\phi_k^-}    
\big)_{L^2([0,\pi]; dx)}\chi_k^+(x),
\quad x \in [0,\pi], \; k \in \bbN_m'',    \no 
\end{align}
form two biorthogonal systems in $L^2([0,\pi]; dx)$. Here 
$\phi_k^{\pm}$, $k\in\bbN_m''$, are root vectors of $H^P$, and 
$\chi_k^{\pm}$, $k\in\bbN_m''$, are root vectors of $\big(H^P\big)^*$, 
with eigenvalues $\xi_{2k}^{\pm}$ and $\ol{\xi_{2k}^{\pm}}$,
respectively. 

We summarize the results of this preparatory section as follows:

\begin{theorem} \lb{t3.1}
Assume $V \in L^2([0,\pi]; dx)$. Then the system 
$\{\phi_k^{\pm}(\cdot)\}_{k\in\bbN}$, as defined in \eqref{3.2}, \eqref{3.14}, 
\eqref{3.20}, \eqref{3.26}, \eqref{3.27},
and \eqref{3.35} are all the root vectors of $H^P$, and the system 
$\{\chi_k^{\pm}(\cdot)\}_{k\in\bbN}$, as defined in \eqref{3.2}, \eqref{3.14}, 
\eqref{3.20}, \eqref{3.26}, \eqref{3.27}, and \eqref{3.35} are all the root 
vectors of $\big(H^P\big)^*$. In particular, $\{\phi_k^{\pm}(\cdot)\}_{k\in\bbN}$  
and $\{\chi_k^{\pm}(\cdot)\}_{k\in\bbN}$ are biorthogonal, complete, and 
minimal in $L^2([0,\pi]; dx)$.
\end{theorem}

We note that associated with the system of root vectors of every 
operator $H(t)$ as used in Theorem \ref{t3.1}, there exists a unique biorthogonal system of root vectors of the operator $H(t)^*$ (also used in Theorem \ref{t3.1}), implying minimality of the system of root vectors of $H(t)$, $t \in [0, 2\pi)$.

\section{The Proof of Theorem \ref{t1.2}} \lb{s4}

Given the preparations in Sections \ref{s2} and \ref{s3}, we now provide 
the proof of Theorem \ref{t1.2} in this section.  

We will apply the following standard criterion for the existence of a 
Riesz basis:

\begin{theorem} [\cite{GK69}, Theorem\ IV.2.1] \lb{t4.1}
Let $\cH$ be a complex separable Hilbert space and $f_k \in \cH$, $k\in \bbN$.
Then the system $\{f_k\}_{k\in\bbN}$ is a Riesz basis in $\cH$ if and only if 
$\{f_k\}_{k\in\bbN}$ is complete in $\cH$ and there exists a corresponding  
complete biorthogonal system $\{g_k\}_{k\in\bbN}$ $($i.e., 
$(f_j,g_k)_{\cH} = \delta_{j,k}$, $j,k\in\bbN$$)$ such that for some $C>0$,
\begin{equation}
\sum_{k\in\bbN} |(f_k,f)_{\cH}|^2 \leq C \|f\|_{\cH}^2, \quad 
\sum_{k\in\bbN} |(g_k,f)_{\cH}|^2 \leq C \|f\|_{\cH}^2, \quad f \in \cH.   \lb{4.1}
\end{equation}
\end{theorem}

By a result of Gel'fand on convex functionals (cf.\ \cite[Sect.\ 21]{AG81}) it actually 
suffices to replace \eqref{4.1} by 
\begin{equation}
\sum_{k\in\bbN} |(f_k,f)_{\cH}|^2 < \infty, \quad 
\sum_{k\in\bbN} |(g_k,f)_{\cH}|^2 < \infty, \quad f \in \cH.   \lb{4.2}
\end{equation}

We start with the necessity part of the proof of Theorem \ref{t1.2}:

\begin{proof}[Proof of necessity of the conditions in Theorem \ref{t1.2}] 
Suppose that $\{F_k^{\pm}\}_{k\in\bbN}$ represents a system of root functions 
of $H^P$ that forms a Riesz basis in $L^2([0,\pi]; dx)$. Since by hypothesis 
$\{F_k^{\pm}\}_{k\in\bbN}$ forms a basis in $L^2([0,\pi]; dx)$, the 
corresponding biorthogonal system $\{G_k^{\pm}\}_{k\in\bbN}$ is unique 
(cf., e.g., \cite[Sect.\ VI.1]{GK69}), and 
hence it is formed by the root vectors of $\big(H^P\big)^*$.

For $k \in \bbN_s$, the eigenvalues $\xi_{2k}^+$ and $\xi_{2k}^-$ are simple 
and the corresponding root subspaces are one-dimensional. Thus, there 
exists a sequence $\{\alpha_k^{\pm}\}_{k\in\bbN_s}$ such that 
\begin{equation}
F_k^{\pm} (x) = \alpha_k^{\pm} \phi_k^{\pm} (x), \quad 
G_k^{\pm} (x) = \big(\alpha_k^{\pm}\big)^{-1} \chi_k^{\pm} (x), \quad  
x \in [0,\pi], \; k \in \bbN_s,     \lb{4.3}
\end{equation}
with $\phi_k^{\pm} (\cdot)$, $\chi_k^{\pm} (\cdot)$, $k \in \bbN_s$, as defined 
in \eqref{3.14}, and \eqref{3.20}.  

According to Definition \ref{d1.1} of a Riesz basis, there exists a constant $C>0$ 
such that 
\begin{equation}
C^{-1} \leq \|F_k^{\pm}\|_{L^2([0,\pi]; dx)} \leq C, \quad 
C^{-1} \leq \|G_k^{\pm}\|_{L^2([0,\pi]; dx)} \leq C, \quad k \in \bbN_s.  \lb{4.4}
\end{equation}
Thus,  
\begin{align}
C^{-1} &\leq \|F_k^{\pm}\|_{L^2([0,\pi]; dx)} 
= |\alpha_k^{\pm}| \|\phi_k^{\pm}\|_{L^2([0,\pi]; dx)} 
= |\alpha_k^{\pm}| \|\chi_k^{\pm}\|_{L^2([0,\pi]; dx)}    \no \\
& = |\alpha_k^{\pm}|^2 \|G_k^{\pm}\|_{L^2([0,\pi]; dx)} 
\leq C |\alpha_k^{\pm}|^2,    \lb{4.5}
\end{align}
that is, $|\alpha_k^{\pm}| \geq C^{-1}$. Replacing $F_k^{\pm}$ by 
$G_k^{\pm}$, one obtains $|\alpha_k^{\pm}|^{-1} \geq C^{-1}$, and hence, 
\begin{equation}
C^{-1} \leq |\alpha_k^{\pm}| \leq C, \quad k \in \bbN_s.    \lb{4.6}
\end{equation}
Consequently, the system
\begin{equation}
\{F_0, F_k^{\pm}\}_{k \in \bbN_{k_0}} \cup \{\phi_k^{\pm}\}_{k \in \bbN_s} 
\cup \{F_k^{\pm}\}_{k \in \bbN_m}    \lb{4.7}
\end{equation}
is a Riesz basis in $L^2([0,\pi]; dx)$ with corresponding complete biorthogonal 
system
\begin{equation}
\{G_0, G_k^{\pm}\}_{k \in \bbN_{k_0}} \cup \{\chi_k^{\pm}\}_{k \in \bbN_s} 
\cup \{G_k^{\pm}\}_{k \in \bbN_m}.     \lb{4.8}
\end{equation}

By Theorem \ref{t4.1}, there exists a constant $C>0$ such that 
\begin{equation}
\sum_{k \in \bbN_s} |(\phi_k^{\pm}, f)_{L^2([0,\pi]; dx)}|^2 \leq C 
\|f\|_{L^2([0,\pi]; dx)}^2, \quad f \in L^2([0,\pi]; dx).     \lb{4.9}
\end{equation}

Next, for $k \in \bbN_s'$, we introduce 
\begin{equation}
\gamma_k^{\pm} = 
- \f{\big(\ol{c(\xi_{2k}^{\pm},\cdot)}, s(\xi_{2k}^{\pm},\cdot)\big)_{L^2([0,\pi]; dx)}}
{\big(\ol{s(\xi_{2k}^{\pm},\cdot)}, s(\xi_{2k}^{\pm},\cdot)\big)_{L^2([0,\pi]; dx)}}   
\lb{4.10}
\end{equation}
and 
\begin{equation}
f_k^{\pm} (x) = c(\xi_{2k}^{\pm},x) + \gamma_k^{\pm} s(\xi_{2k}^{\pm},x), 
\quad x \in [0,\pi], \; k \in \bbN_s',     \lb{4.11}
\end{equation}
implying 
\begin{equation}
\big(\ol{f_k^{\pm}}, s(\xi_{2k}^{\pm},\cdot)\big)_{L^2([0,\pi]; dx)} = 0, 
\quad k \in \bbN_s'.     \lb{4.12}
\end{equation}

By \eqref{2.10}, \eqref{2.11}, and \eqref{2.15} one thus infers that 
\begin{align}
& \gamma_k^{\pm} \underset{k \to \infty}{=} \oh(1),    \lb{4.13} \\
& \|c(\xi_{2k}^{\pm},\cdot)\|_{L^2([0,\pi]; dx)}^2 \underset{k \to \infty}{=} 
[1 + \oh(1)] \pi/2,    \lb{4.14} \\
& (c(\xi_{2k}^{\pm},\cdot), \xi_{2k}^{\pm} s(\xi_{2k}^{\pm},\cdot))_{L^2([0,\pi]; dx)} 
\underset{k \to \infty}{=} \oh(1),    \lb{4.15} \\
& \|f_k^{\pm}\|_{L^2([0,\pi]; dx)}^2 \underset{k \to \infty}{=} 
[1 + \oh(1)]\pi/2.    \lb{4.16} 
\end{align}
Thus, by \eqref{3.14},
\begin{equation}
|(f_k^{\pm}, \phi_k^{\pm})_{L^2([0,\pi]; dx)}|^2 \underset{k \to \infty}{=} 
\big|\xi_{2k}^{\pm} s(\xi_{2k}^{\pm},\pi)/u_+^{\bullet}(\xi_{2k}^{\pm})\big| 
[1 + \oh(1)] \pi /2,    \lb{4.17}
\end{equation}
and by \eqref{4.1}, 
\begin{equation} 
\big|\xi_{2k}^{\pm} s(\xi_{2k}^{\pm},\pi)/u_+^{\bullet}(\xi_{2k}^{\pm})\big| 
\leq C, \quad k \in \bbN_s'.    \lb{4.18} 
\end{equation}

For $k \in \bbN_s''$ we define $\gamma_k^-$ as in \eqref{4.10} and obtain the estimate \eqref{4.18} for $\xi_{2k}^-$ using \eqref{3.20}. The corresponding estimate for $\xi_{2k}^+$ is trivial as $s(\xi_{2k}^+,\pi) = 0$. Thus one 
concludes that 
\begin{equation} 
\big|\xi_{2k}^{\pm} s(\xi_{2k}^{\pm},\pi)/u_+^{\bullet}(\xi_{2k}^{\pm})\big| 
\leq C, \quad k \in \bbN_s.    \lb{4.19} 
\end{equation}

Employing \eqref{4.19} and the estimates,
\begin{align}
& \big|\xi_{2k}^{\pm} s(\xi_{2k}^{\pm},\pi)\big| \geq \big|\xi_{2k}^{\pm} - 
\zeta_{2k}\big|
\min_{|\zeta - 2k| \leq 10^{-1}} \big|\zeta s(\zeta,\pi) /
[\zeta - \zeta_{2k}]\big|
\geq C \big|\xi_{2k}^{\pm} - \zeta_{2k}\big|,    \lb{4.20} \\
& \big|u_+^{\bullet} (\xi_{2k}^{\pm})\big| \leq 2
|u_+^{\bullet \bullet}(\omega_{2k})
[\xi_{2k}^{\pm} - \omega_{2k}]| \leq
C \big|\xi_{2k}^+ - \xi_{2k}^- \big|,     \lb{4.21}
\end{align} 
one arrives at \eqref{1.19}. 
\end{proof}

We emphasize that the set $\bbN_m$ plays no role in condition 
\eqref{1.19} in Theorem \ref{t1.2}. 

Next, we turn to the sufficiency part of the proof of Theorem \ref{t1.2}:

We start with the following result:

\begin{lemma} \lb{l4.2}
Assume condition \eqref{1.19}, that is, 
\begin{equation}
\sup_{\substack{k \in \bbN, \\ \lambda_{2k}^+ \neq \lambda_{2k}^-}} 
\f{|\mu_{2k} - \lambda_{2k}^{\pm}|}{|\lambda_{2k}^+ - \lambda_{2k}^-|} < \infty.   
\lb{4.22}
\end{equation}
Then
\begin{equation}
\sup_{k \in \bbN_s} \bigg|\f{u_-(\xi_{2k}^{\pm})}{u_+^{\bullet}(\xi_{2k}^{\pm})}\bigg| 
\leq C < \infty.    \lb{4.23}
\end{equation}
\end{lemma}
\begin{proof}
At this point we need to use the enumeration \eqref{3.4b} of all the zeros of 
$s(\cdot,\pi)$, $u_+(\cdot) - 1$, and $u_+^\bullet(\cdot)$, respectively, 
\begin{equation}
\{\zeta_k\}_{k\in\bbZ\backslash\{0\}}, \quad \{\xi_k^\pm\}_{k\in\bbZ},
\quad \{\omega_k\}_{k\in\bbZ}. 
\end{equation}
For all sufficiently large $|k|$ we have 
\begin{align}
\big|u_-(\zeta_k)^2\big| &= \big|u_+(\zeta_k)^2 - 1\big| 
= \big|u_+(\zeta_k)^2 - u_+(\xi_k^+)^2\big| 
\leq 3 |u_+(\zeta_k) - u_+(\xi_k^+)|   \no \\
& \leq 3 \bigg|\int_{\zeta_k}^{\xi_k^+} d\zeta \, u_+^{\bullet}(\zeta)\bigg| 
\leq 3 \bigg|\int_{\zeta_k}^{\xi_k^+} d\zeta 
\int_{\omega_k}^{\zeta} d \zeta' \, u_+^{\bullet \bullet}(\zeta')\bigg|    \no \\
& \leq C \big|\xi_k^+ - \zeta_k\big| 
\max \{|\omega_k - \xi_k^+ |, |\omega_k - \zeta_k|\}.  
\lb{4.24}
\end{align}
Since by \eqref{2.11} and \eqref{2.12},
\begin{equation}
\big\{k \big[\xi_k^+ - \zeta_k\big]\big\}_{k \in \bbZ\backslash\{0\}} 
\in \ell^2(\bbZ\backslash\{0\}), \quad 
\{k [\omega_k - \xi_k^+]\}_{k \in \bbZ} \in \ell^2(\bbZ), 
\lb{4.25}
\end{equation}
one obtains that 
\begin{equation}
\{\zeta_k u_-(\zeta_k)\}_{k \in \bbZ\backslash\{0\}} 
\in \ell^2(\bbZ\backslash\{0\}).     \lb{4.26}
\end{equation}

Because of the representation \eqref{2.12}, the function $\zeta s(\zeta,\pi)$ 
has all properties of a function of sine-type (cf.\ \cite[Lecture\ 22]{Le96}), except that it may possess a finite number of multiple zeros. Thus, following the 
methods in \cite[Lecture\ 22]{Le96}, one can prove that the entire 
function $\zeta u_-(\zeta)$ is representable by the Lagrange--Hermite 
interpolation series (cf.\ \cite{GLO97}) as 
\begin{equation}
\zeta u_-(\zeta) = \sum_{k \in \bbZ} {\rm Res}_{\zeta' = \zeta_k} 
\bigg\{\f{\zeta s(\zeta,\pi) - \zeta' s(\zeta',\pi)}{(\zeta-\zeta') \zeta' s(\zeta',\pi)} 
\zeta' u_-(\zeta')\bigg\}   \lb{4.27} 
\end{equation}
(introducing $\zeta_0 =0$), convergent in the norm of the Paley--Wiener 
space ${\mathbb{PW}_\pi}$. In this context we recall that the Paley--Wiener 
class ${\mathbb{PW}_\pi}$ is defined as the set of all entire functions of exponential type not exceeding $\pi$ satisfying
\begin{equation}
\|f\|^2_{\mathbb{PW}_\pi}=\int_{\R} dx\, |f(x)|^2  <\infty.    \lb{4.27a}
\end{equation}
Since the function $\zeta s(\zeta,\pi)$ may now have a finite number of 
multiple zeros we will add some more remarks concerning the representation  
\eqref{4.27} at the end of this proof and for now assume its validity.

Asymptotically, the zeros of $\zeta s(\cdot,\pi)$ are 
simple and hence for $|k|\in\bbN$ sufficiently large, the $k$th term 
under the sum in \eqref{4.27} is of the form 
\begin{equation}
\f{\zeta s(\zeta,\pi)}{(\zeta - \zeta_k) \zeta_k s^{\bullet}(\zeta_k,\pi)}
\zeta_k u_-(\zeta_k).    \lb{4.28} 
\end{equation}  
Consequently, 
\begin{equation}
\zeta u_-(\zeta) \underset{\zeta \to \zeta_k}{=} \zeta_k u_-(\zeta_k)
[1 + \oh(\zeta - \zeta_k)] + \zeta s(\zeta,\pi) \rho_k(\zeta),    \lb{4.29}
\end{equation}
where 
\begin{equation}
\rho_k(\zeta) = \sum_{j \in \bbN\backslash\{k\}} {\rm Res}_{\zeta' = \zeta_j} 
\bigg\{\f{\zeta' u_-(\zeta')}{(\zeta-\zeta') \zeta' s(\zeta',\pi)}\bigg\}.    \lb{4.30}
\end{equation}
Taking $\zeta = \xi_{2k}^{\pm}$ in \eqref{4.30} then yields for $k\in\bbN$ 
sufficiently large,
\begin{equation}
\big|u_-(\xi_{2k}^{\pm})\big| \leq C |u_-(\zeta_{2k})| 
+ c_k \big|\xi_{2k}^{\pm} s(\xi_{2k}^{\pm},\pi)\big|,    \lb{4.31}
\end{equation}
with 
\begin{equation}
c_k \underset{k\to\infty}{=} \oh(1).     \lb{4.31a}
\end{equation}
Thus, 
\begin{equation}
\bigg|\f{u_-(\xi_{2k}^{\pm})}{u_+^{\bullet} (\xi_{2k}^{\pm})}\bigg| \leq C 
\Bigg[\f{\big[|\xi_{2k}^+ - \zeta_{2k}| |\omega_{2k} 
- \xi_{2k}^+|\big]^{1/2}}{|\xi_{2k}^+ - \xi_{2k}^-|} 
+ \f{\xi_{2k}^{\pm} s(\xi_{2k}^{\pm},\pi)}{\big|\xi_{2k}^+ - \xi_{2k}^-\big|}\Bigg], 
\lb{4.32}
\end{equation}
and hence \eqref{1.19} implies \eqref{4.23}. 

Returning to the Lagrange--Hermite interpolation series \eqref{4.27}, we note 
that alternatively to following the methods in \cite[Lecture\ 22]{Le96}, which ultimately yields convergence of \eqref{4.27} in the ${\mathbb{PW}_\pi}$-norm, one can introduce a new function $\sigma(\cdot)$ that has the appropriate 
number of simple zeros in a small neighborhood of the multiple zeros of 
$\zeta s(\zeta,\pi)$ and otherwise the same simple zeros as the latter, and 
writes \eqref{4.27} with $\zeta s(\zeta,\pi)$ replaced by $\sigma(\zeta)$.The 
absolute convergence of the sum on the right-hand side of \eqref{4.27}, which suffices for our purpose, then follows from the fact \eqref{4.31a} together with 
\eqref{2.11}, \eqref{4.26}, and \eqref{4.27} in terms of $\sigma(\cdot)$, which permits the limiting procedure from $\sigma(\zeta)$ to $\zeta s(\zeta,\pi)$. 
\end{proof}

\begin{proof}[Proof of sufficiency of the conditions in Theorem \ref{t1.2}] 
Since the root systems of the operators $H^P$ and $\big(H^{P}\big)^*$ 
are complete in $L^2([0,\pi]; dx)$ by Theorem \ref{t2.1}, it suffices to 
prove that the systems 
\begin{equation}
\{\phi_k^{\pm}\}_{k \in \bbN_s \cup \bbN_m} \, \text{ and } \, 
\{\chi_k^{\pm}\}_{k \in \bbN_s \cup \bbN_m}    \lb{4.33}
\end{equation}
constructed in Section \ref{s3} satisfy the conditions \eqref{4.1} in 
Theorem \ref{t4.1}. To this end one observes that every function
$\phi_k^{\pm}(x)$, $x \in [0,\pi]$, $k \in \bbN_s'$, is a linear combination 
of the functions
\begin{equation}
c(\xi_{2k}^{\pm},x), \quad c^{\bullet}(\xi_{2k}^{\pm},x), \quad
\xi_{2k}^{\pm} s(\xi_{2k}^{\pm},x), \quad
\xi_{2k}^{\pm} s^{\bullet}(\xi_{2k}^{\pm},x), \quad x \in [0,\pi],
\lb{4.34}
\end{equation}
for which the inequalities in \eqref{4.1} are satisfied. In this context we 
note that the latter functions are of one of the
following forms,
\begin{equation}
\cos(2kx) + k^{-1} f_k(x), \quad \sin(2kx) + k^{-1} g_k(x), \quad x \in 
[0,\pi], \; k \in \bbN,
\end{equation}
with
\begin{equation}
\sup_{k\in\bbN} \big[\|f_k\|_{L^2([0,\pi];dx)} + 
\|g_k\|_{L^2([0,\pi];dx)}\big] < \infty,
\end{equation}
and hence are parts of a Riesz basis in $L^2([0,\pi];dx)$. Consequently, it 
suffices
to verify that the coefficients in these linear combinations remain bounded 
as $k\to\infty$.

For $k\in\bbN_s'$, these coefficients are either given by 
\begin{equation}
\bigg[\f{\xi_{2k}^{\pm} s(\xi_{2k}^{\pm},\pi)}
{u_+^{\bullet}(\xi_{2k}^{\pm})}\bigg]^{1/2} \, \text{ and } \, 
\bigg[\f{u_-(\xi_{2k}^{\pm})^2}
{\xi_{2k}^{\pm} s(\xi_{2k}^{\pm},\pi) 
u_+^{\bullet}(\xi_{2k}^{\pm})}\bigg]^{1/2}, \quad k \in \bbN_s',   \lb{4.35}
\end{equation}
or their complex conjugates. The first coefficient in \eqref{4.35} is bounded since 
\begin{equation}
\f{\xi_{2k}^{\pm} s(\xi_{2k}^{\pm},\pi)}{u_+^{\bullet}(\xi_{2k}^{\pm})} 
\underset{k\to\infty}{=} 
\f{2 \xi_{2k}^{\pm} s^{\bullet}(\zeta_{2k},\pi) (\xi_{2k}^{\pm} - \zeta_{2k})}
{u_+^{\bullet \bullet}(\omega_{2k}) (\xi_{2k}^+ - \xi_{2k}^-)}[1 + \oh(1)]    \lb{4.36}
\end{equation} 
and \eqref{1.19} holds. In this context we recall that 
\begin{equation}
u_+^{\bullet\,\bullet}(\omega_k) \underset{k\to\infty}{=} - \pi^2 + \oh(1). \lb{4.36a}
\end{equation}

To estimate the second coefficient in \eqref{4.35} we 
use the representation \eqref{4.29} in the form
\begin{equation}
u_-(\zeta) \underset{k \to \infty}{=} u_-(\zeta_{2k}) [1 + \oh(1)]  
+ \zeta s(\zeta,\pi) \oh(1),    \lb{4.37} 
\end{equation}
uniformly with respect to $\zeta$ in a sufficiently small disk around 
$\xi_{2k}^{\pm}$ of fixed (i.e., $k$-independent) radius. Equation \eqref{4.37} 
then implies 
\begin{equation}
|u_-(\xi_{2k}^\pm)|\leq 2 |u_-(\zeta_{2k})| +|\xi_{2k}^\pm s(\xi_{2k}^\pm ,\pi)|.
\lb{4.38}
\end{equation}
Moreover, 
\begin{align}
\big|u_-(\zeta_{2k})^2\big| &= \big|1 - u_+(\zeta_{2k})^2\big| 
\leq 3 |u_+(\xi_{2k}^{\pm}) - u_+(\zeta_{2k})| 
\leq 3 \bigg|\int_{\zeta_{2k}}^{\xi_{2k}^{\pm}} d\zeta \, u_+^{\bullet \bullet}(\zeta) 
(\xi_{2k}^{\pm} - \zeta)\bigg|    \no \\
& \leq C |\xi_{2k}^{\pm} - \zeta_{2k}|^2.    \lb{4.39}
\end{align}
Combining Lemma \ref{l4.2} with the estimates \eqref{4.38} and \eqref{4.39} 
one also concludes that the second coefficient in \eqref{4.35} is bounded for 
$k \in \bbN_s'$. 

For $k\in\bbN_s''$, the coefficient 
$\big[-2 /[\xi_{2k}^+ s^{\bullet}(\xi_{2k}^+,\pi)]\big]^{1/2}$ in $\phi_k^+$ 
multiplying $\xi_{2k}^+ s(\xi_{2k}^+,\cdot)$ in 
\eqref{3.20} is bounded by \eqref{2.11}. Regarding $\phi_k^-$, $k\in\bbN_s''$, 
its coefficients are bounded as in the case $k\in\bbN_s'$, treated above. 

Hence, it remains to consider the coefficients of $\phi_k^{\pm}$, represented 
as linear combinations of the functions in \eqref{4.34}, for $k \in\bbN_m$. In this 
context it suffices to observe that for $k\in\bbN_m$, one has 
\begin{equation}
s(\xi_{2k}^{\pm},\pi) \underset{k\to\infty}{=} s^{\bullet}(\zeta_{2k})
\big(\xi_{2k}^{\pm} - \zeta_{2k}\big) [1 + \oh(1)],    \lb{4.40}
\end{equation} 
and 
\begin{equation}
|u_-(\xi_{2k}^{\pm})| \leq |u_-(\zeta_{2k})| + |u_-(\xi_{2k}^{\pm}) - u_-(\zeta_{2k})| 
\leq C \big|\xi_{2k}^{\pm} - \zeta_{2k}\big|.     \lb{4.41}
\end{equation}
Thus, the fraction 
$u_-(\xi_{2k}^{\pm})\big/\big[\xi_{2k}^{\pm} s(\xi_{2k}^{\pm},\pi)\big]$, 
multiplying $\xi_{2k}^{\pm} s(\xi_{2k}^{\pm},\cdot)$ 
in \eqref{3.33} and \eqref{3.35}, is bounded with respect 
to $k\in\bbN_m$. The boundedness of the remaining coefficients in the linear 
combinations representing $\phi_k^{\pm}$ and $\chi_k^{\pm}$ is evident from 
our considerations thus far, in particular, 
$|u_+^{\bullet \, \bullet} (\xi_{2k}^\pm)|$ in \eqref{3.25a} is bounded from 
below by \eqref{4.36a}. This completes the proof of the sufficiency part of 
Theorem \ref{t1.2}. 
\end{proof}

\section{The Proof of Theorem \ref{t1.4}} \lb{s5}

In our final section we prove the Schauder basis results in connection with 
$L^p([0,\pi];dx)$, $p \in (1,\infty)$,  stated in Section \ref{s1}.  

Let $\cB$ denote a complex, separable Banach space and denote by $\cB^*$ its 
conjugate dual space. We recall that a system of vectors 
$\{g_k\}_{k\in\bbN} \subset \cB$ 
is called {\it complete} in $\cB$ if $\ol{\text{lin.\,span} \, \{g_k\}_{k \in \bbN}} = \cB$. 
Moreover (as in the Hilbert space context), a system $\{h_k\}_{k \in \bbN} \subset \cH$ 
is called {\it minimal} in $\cB$ if no vector $h_{k_0} \in \{h_k\}_{k \in \bbN}$ satisfies 
$h_{k_0} \in \ol{\text{lin.\,span} \, \{h_k\}_{k \in \bbN\backslash \{k_0\}}}$. 

A system $\{h_k, \ell_k\}_{k \in\bbN}$, $h_j \in\cB$, $\ell_j \in \cB^*$, $j\in\bbN$, is called 
{\it biorthogonal} if
\begin{equation}
\ell_j(h_k) =0, \quad j \neq k, \; j,k \in \bbN,
\end{equation}
and {\it biorthonormal} if 
\begin{equation}
\ell_j(h_k) = \delta_{j,k}, \quad j,k \in \bbN. 
\end{equation}
The system $\{\ell_k\}_{k \in\bbN}$ is then called {\it biorthonormal} (or {\it dual\,}) to 
$\{h_k\}_{k \in\bbN}$. 
In general, such a biorthonormal system $\{\ell_k\}_{k \in\bbN}$ is nonunique. However, if 
$\{h_k\}_{k \in\bbN}$ is complete in $\cB$, then its biorthogonal $\{\ell_k\}_{k \in\bbN}$ 
is unique, if it exists. In this context we also mention that a given system 
$\{g_k\}_{k\in\bbN} \subset \cB$ has a biorthonormal system 
$\{\ell_k\}_{k \in\bbN} \in \cB^*$ if and only if $\{g_k\}_{k\in\bbN}$ is minimal.

As in the Hilbert space context considered in the bulk of this paper, the 
system $\{f_k\}_{k\in\bbN} \subset \cB$ is called a {\it Schauder basis} in $\cB$
\begin{align}
\begin{split}
& \text{if for each $f \in \cB$, there exists unique $c_k = c_k(f) \in\bbC$, $k \in \bbN$, 
such that}     \lb{6.1} \\
& \quad \text{$f = \sum_{k \in \bbN} c_k(f) f_k$ converges in the norm of $\cB$.} 
\end{split} 
\end{align} 

One recalls (cf., \cite[Sects.\ 1.2,\,1.4]{KS89}) that $\{f_k\}_{k\in\bbN} \subset \cB$ is a Schauder 
basis in $\cB$, if and only if the following three conditions hold:
\begin{align}
& (i) \;\;\, \{f_k\}_{k\in\bbN} \, \text{ is complete in $\cB$,}     \\
& (ii) \;\, \{f_k\}_{k\in\bbN} \, \text{ is minimal in $\cB$,}     \\
& (iii) \; \text{then there exists a unique biorthonormal system 
$\{l_k\}_{k\in\bbN} \subset \cB^*$ and a} \no \\
& \qquad \, \text{constant $C>0$ such that for all $N\in\bbN$,}   \no \\
& \hspace*{8mm} 
\bigg\|\sum_{k=1}^N \ell_k(f) f_k\bigg|_{\cB} \leq C \|f\|_{\cB}, \quad f \in \cB. 
\lb{6.x} 
\end{align}
 
In the following, for $f \in L^p([0,\pi];dx)$ and $g \in L^q([0,\pi]; dx)$, with 
$p, q \in (1,\infty)$, $(1/p) + (1/q) = 1$, we introduce the functional 
\begin{equation}
g(f) = \int_0^{\pi} dx\, \ol{g(x)} f(x), 
\end{equation}  
linear with respect to $f$ and antilinear in $g$.  

Finally, we recall that if $\{\psi_k(\cdot)\}_{k\in\bbN}$ is a Schauder basis in 
$L^p([0,\pi]; dx)$, $p \in (1,\infty)$, then its biorthonormal system 
$\{\eta_k(\cdot)\}_{k\in\bbN}$ is a basis in $L^q([0,\pi]; dx)$, where $(1/p)+(1/q)=1$. 
 
\begin{proof}[Proof of necessity of the conditions in Theorem \ref{t1.4} 
for $1<p\leq 2$.] ${}$ \\
In analogy to the case $p=2$, if there exists a Schauder basis of root vectors of 
$H^P$ in $L^p([0,\pi]; dx)$, then the system 
$\{\phi_k^{\pm}\}_{k\in\bbN_s} \subset L^p([0,\pi]; dx)$ is a part 
of a Schauder basis of $H^P$ as well, with corresponding biorthogonal system 
$\{\chi_k^{\pm}\}_{k\in\bbN_s} \subset L^q([0,\pi]; dx)$. Thus, by \eqref{6.x} 
there exists a constant $C>0$ such that 
\begin{equation}
\|\chi_k^{\pm} (f) \phi_k^{\pm}\|_{L^p([0,\pi]; dx)} \leq C \|f\|_{L^p([0,\pi]; dx)}, \quad 
f \in L^p([0,\pi]; dx), \; k \in \bbN_s,    \lb{6.y}    
\end{equation}
holds.

For $k\in\bbN_s'$, we still define the functions $f_k^{\pm}$ by \eqref{4.10}, 
\eqref{4.11} and find that \eqref{4.13} and \eqref{4.15} continue to hold. 
However, instead of \eqref{4.14} one now obtains for some constant $C_r>0$, 
\begin{align}
& \|c(\xi_{2k}^{\pm},\cdot)\|_{L^r([0,\pi]; dx)} 
= \bigg(\int_0^{\pi} dx \, |\cos(\xi_{2k}^{\pm} x)|^r\bigg)^{1/r} + \Oh(k^{-1})  \no \\
& \quad 
= \bigg(\int_0^{\pi} dx \, |\cos(2k x)|^r\bigg)^{1/r} + \Oh(k^{-1}) \leq C_r,  
\quad r \in (1,\infty),    \\
& \|\xi_{2k}^{\pm} s(\xi_{2k}^{\pm},\cdot)\|_{L^r([0,\pi]; dx)} \leq C_r, \quad r \in (1,\infty). 
\end{align}
An application of \eqref{4.10} and \eqref{4.11}  then yields 
\begin{equation}
\|f_k^{\pm}\|_{L^r([0,\pi]; dx)} \leq C_r, \quad r \in (1,\infty),
\end{equation}
which permits one to consider $f_k^{\pm}$ as elements of 
$L^q([0,\pi]; dx) = L^p([0,\pi]; dx)^*$. By \eqref{6.y} one concludes that 
\begin{align}
& \big|\chi_k^{\pm} (f_k^{\pm}) \ol{f_k^{\pm}} (\phi_k^{\pm})\big| 
= \big|\ol{f_k^{\pm}} (\chi_k^{\pm} (f_k^{\pm})\phi_k^{\pm})\big|  \no \\
& \quad \leq C \|\chi_k^{\pm} (f_k^{\pm}) \phi_k^{\pm}\|_{L^p([0,\pi]; dx)} 
\|f_k^{\pm}\|_{L^q([0,\pi]; dx)} \leq C.
\end{align}
Hence one obtains
\begin{align}
C &\geq \bigg|\int_0^{\pi} dx \, \ol{f_k^{\pm}(x)} \chi_k^{\pm} (x)\bigg|
\bigg|\int_0^{\pi} dx \, f_k^{\pm}(x) \phi_k^{\pm} (x)\bigg|   \no \\
& = |\xi_{2k}^{\pm} s(\xi_{2k}^{\pm},x)/ u_+^{\bullet} (\xi_{2k}^{\pm})|
\bigg|\int_0^{\pi} dx \, f_k^{\pm}(x) c(\xi_{2k}^{\pm},x)\bigg|    \no \\
& \geq D |\xi_{2k}^{\pm} s(\xi_{2k}^{\pm},x)/ u_+^{\bullet} (\xi_{2k}^{\pm})| \lb{6.15}
\end{align}
for some constant $D>0$. This proves \eqref{4.18}. The same arguments yield 
 \eqref{6.15} for $k \in \bbN_s''$. To complete the proof of 
necessity of the conditions in Theorem \ref{t1.4} for $1<p\leq 2$ it now suffices 
to employ \eqref{4.20}, \eqref{4.21}.
\end{proof}

\begin{proof}[Proof of sufficiency of the conditions in Theorem \ref{t1.4} 
for $1<p\leq 2$.] ${}$ \\
We need to prove that for all $f \in L^p([0,\pi]; dx)$, 
\begin{equation}
f(\cdot) = \sum_{k\in\bbN_0} \bigg(\int_0^{\pi} dy \, \ol{\chi_k^{\pm}(y)} f(y)\bigg)
\phi_k^{\pm}(\cdot),   \lb{6.z}
\end{equation}
converges in the norm of $L^p([0,\pi]; dx)$. 

We showed in Section \ref{s4}, that $\phi_k^{\pm}$ and $\chi_k^{\pm}$ are linear combinations of the functions in \eqref{4.34} with coefficients bounded as 
$k\to\infty$. The standard Volterra integral equations
\begin{align}
c(\zeta,x) &= \cos(\zeta x) + \int_0^x dx' \, \zeta^{-1} \sin(\zeta (x-x')) V(x') c(\zeta,x'), \\
\zeta s(\zeta,x) &= \sin(\zeta x) + \int_0^x dx' \, \zeta^{-1} \sin(\zeta (x-x')) V(x') 
\zeta s(\zeta,x'), \\
& \hspace*{5.4cm} \zeta \in \bbC, \; x \in [0,\pi],   \no 
\end{align}
combined with the asymptotic formulas \eqref{3.4a} show that each function in 
\eqref{4.34} is of the form
\begin{equation}
x^{\alpha} \cos(2kx) + \f{\tau_{k,\alpha}^{\pm}(x)}{k}, \quad 
x^{\beta} \sin(2kx) + \f{\rho_{k,\beta}^{\pm}(x)}{k},  \quad \alpha, \beta \in \{0,1\}, 
\; x \in [0,\pi],  
\end{equation}
where
\begin{equation}
\sup_{k \in \bbN, \, \alpha, \beta \in \{0,1\}} \sup_{x\in [0,\pi]}
\big[\big|\tau_{k,\alpha}^{\pm}(x)\big| + \big|\rho_{k,\beta}^{\pm}(x)\big|\big] < \infty.
\end{equation}

Consequently, \eqref{6.z} can be split into finitely many terms of the following type:
\begin{align}
& x^\beta \sum_{k\in\bbN} c_k^{(1)} \bigg(\int_0^{\pi} dy \, y^{\alpha} f(y) 
a_{k,1}(y)\bigg) b_{k,1}(x),    \lb{6.a} \\
& x^\beta \sum_{k\in\bbN} c_k^{(2)} \f{1}{k} \bigg(\int_0^{\pi} dy \, f(y) 
c_{k,2}(y)\bigg) b_{k,2}(x),    \lb{6.b} \\
& \sum_{k\in\bbN} c_k^{(3)} \f{1}{k} \bigg(\int_0^{\pi} dy \, y^{\alpha} f(y) 
a_{k,3}(y)\bigg) c_{k,3}(x),    \lb{6.c} \\
& \sum_{k\in\bbN} c_k^{(4)} \f{1}{k^2}\bigg(\int_0^{\pi} dy \, f(y) 
c_{k,4}(y)\bigg) d_{k,4}(x),    \lb{6.d} 
\end{align}
where $\alpha, \beta \in \{0,1\}$, the functions $a_{k,j}$ and $b_{k,j}$ coincide with 
one of the functions $\cos(2kx)$ and $\sin(2kx)$, and the functions $c_{k,j}$ and 
$d_{k,j}$ coincide with one of the functions $\tau_{k,\alpha}^{\pm}$ and 
$\rho_{k,\beta}^{\pm}$. In addition, as a consequence of \eqref{1.19}, 
\begin{equation}
\{c_k^{(j)}\}_{k\in\bbN} \in \ell^{\infty}(\bbN), \quad 1\leq j \leq 4.
\end{equation}

According to part $(i)$ of the Hausdorff--Young theorem in 
\cite[Theorem\ XII.2.3]{Zy90}), for every $f \in L^p([0,\pi];dx)$, one has
\begin{equation}
\bigg\{\int_0^{\pi} dy \, y^{\alpha} f(y) e^{in y}\bigg\}_{n\in\bbZ} \in \ell^q(\bbZ), 
\quad \f{1}{p} + \f{1}{q} =1.
\end{equation} 
Moreover, it follows from part $(ii)$ of  \cite[Theorem\ XII.2.3]{Zy90} that the estimate 
\begin{align}
\begin{split}
& \bigg\|\sum_{k=m}^n c_k^{(1)} \bigg(\int_0^{\pi} dy \, y^{\alpha} f(y) 
a_{k,1}(y)\bigg) b_{k,1}(\cdot)\bigg\|_{L^p([0,\pi];dx)}    \\
& \quad \leq C\bigg(\sum_{k=m}^n \bigg|\int_0^{\pi} dy \, y^{\alpha} f(y) 
a_{k,1}(y)\bigg|^q\bigg)^{1/q}
\end{split} 
\end{align}
holds for all $m, n \in \bbZ$, and hence the series \eqref{6.a} converges in $L^p([0,\pi];dx)$.

In addition, the series \eqref{6.b} converges in $L^2([0,\pi];dx)$, 
and it is easy to see that \eqref{6.c} as well as \eqref{6.d} converge uniformly to continuous functions on $[0,\pi]$. 

Thus, the series on the right-hand side of \eqref{6.z} converges in $L^p([0,\pi];dx)$ 
and equality in \eqref{6.z} now follows from completeness of the system 
$\{\phi_k^{\pm}\}_{k\in\bbN_0}$ in $L^p([0,\pi];dx)$. 
\end{proof}

\begin{proof}[Proof of Theorem \ref{t1.4} 
for $2 \leq p < \infty$.] ${}$ \\
Then $1 \leq q =p/(p-1) \leq 2$ and 
Theorem \ref{t1.4} applies to the operator $\big(H^P\big)^*$ in the space $L^q([0,\pi];dx)$. 
Since this operator is generated by the complex conjugate potential $\ol{V(\cdot)}$ and 
the periodic and antiperiodic as well as Dirichlet boundary conditions are all self-adjoint, 
condition \eqref{1.19} for $H^P$ and $\big(H^P\big)^*$ coincide. At the same time, the 
system of root vectors of $H^P$ contains a Schauder basis if and only if the system of 
root vectors of $\big(H^P\big)^*$ contains a Schauder basis. Since the system 
$\{\phi_k^{\pm}\}_{k\in\bbN_0}$ is dual to $\{\chi_k^{\pm}\}_{k\in\bbN_0}$, this proves Theorem \ref{t1.4} for $p \in [2,\infty)$.
\end{proof}

\section{Some Remarks} \lb{s6}

In this section we briefly further illustrate the principal result of this paper:

\begin{remark} \lb{r6.1} 
Starting with the pioneering works by Birkhoff and Tamarkin (cf. the
discussion in [60]), almost all results related to eigenfunction expansions generated
by ordinary differential operators were obtained within the framework of direct
spectral theory. For instance, in the papers \cite{DM03a} and \cite{DM09},  necessary 
and sufficient conditions for the Riesz property of systems of eigenfunctions 
were found for classes of two- and four-term trigonometric potentials. These conditions 
are explicitly stated in terms of the coefficients of the polynomials (see also \cite{Ma06b}, 
where the example of a two-term trigonometric potential is discussed near the end). In 
\cite{SV09} a specific system of root vectors corresponding to a generic potential in the space $L^2([0,\pi];dx)$ was introduced, and a criterium for it to form a Riesz basis was proved in terms of the Fourier coefficients of the eigenfunctions (i.e., in somewhat less explicit terms).  

Having in mind the periodic/antiperiodic boundary problems for
Schr\"odinger operators only, we note that most results in this context were obtained
under assumptions which restrict the smoothness properties of potentials,
see, for instance, \cite{DV05}, \cite{KM98}, \cite{Ki09}, \cite{Ku06}, \cite{Ma06a}, 
\cite{Ma06c}, \cite{Ma10}, \cite{MM08}, and \cite{MM10}. Such assumptions are attractive, 
since they are expressed in terms of explicit properties of the potential. However, in
spite of sometimes rather involved eigenfunction constructions, they did not result
in a criterium, that is, necessary and sufficient conditions for the desired basis
property of the root systems of operators with {\it generic} potentials. It is worth noting
that the above-mentioned smoothness restrictions are redundant in connection with 
analytic and $C^k$-potentials, even in the self-adjoint situation where the Riesz property always holds. In sharp contrast to what has just been described, in our approach, we treat the problem of Riesz and Schauder bases within the framework of inverse spectral theory. Our aim was three-fold:

$\bullet$ First, to deal with the Riesz property of the root system of the operators 
$H^P$ and $H^{AP}$ with an arbitrary potential $V \in L^2([0,\pi];dx)$, with no 
restrictions on its form,  no smoothness properties, and no analyticity assumptions.

$\bullet$ Second, to obtain necessary and sufficient conditions for the existence of Riesz and Schauder bases in terms of spectral data which permit one to construct (or reconstruct)  a potential $V$ in a one-to-one manner. 

$\bullet$ Third, to establish the Riesz and property of at least one root system without  restricting ourselves to a specific choice of its elements.
 
The spectral data adequate for our purpose were proposed in \cite{ST96a} and 
\cite{Tk01}. These data consist of the functions $u_+(\cdot)$ and $s(\cdot,\pi)$ 
or, alternatively, of the periodic/antiperiodic spectra 
$\{\lambda_0^+, \lambda_{2k}^+, \, \lambda_{2k}^-\}_{k\in\bbN}$,      
$\{\lambda_{2k+1}^+, \, \lambda_{2k+1}^-\}_{k\in\bbN_0}$, and the Dirichlet spectra,  
$\{\mu_j\}_{j\in\bbN}$, respectively. As shown in \cite{ST96} and \cite{Tk92}, these two 
sets of data represent independent parameters which uniquely determine the potential 
$V  \in L^2([0,\pi];dx)$, in particular, we recall that the precise properties of these two 
sets of data, as implied by the condition $V \in L^2([0,\pi];dx)$, were recorded in  
\cite{ST96} and \cite{Tk92}. 

It follows from \cite{ST96} and \cite{Tk92} that for arbitrary positive sequences 
\begin{equation}
\{\alpha_n\}_{n\in\bbN}, \, \{\beta_n\}_{n\in\bbN} \in \ell^2(\bbN),
\end{equation}  
there exists a potential $V \in L^2([0,\pi];dx)$ such that the conditions 
\begin{equation}
\lim_{n\to\infty}|\lambda_n^+ - \lambda_n^-| \alpha_n^{-1}
= \lim_{n\to\infty}|\lambda_n^- - \mu_n| \beta_n^{-1} = 1 
\end{equation}
are satisfied. As shown in \cite{ST96a} and \cite{Tk01}, this potential may have arbitrary (e.g., fractional) smoothness, or even be an analytic function, while the Riesz basis 
property of its system of eigenfunctions depends on the existence of the finite limit
$\limsup_{n\to\infty}\, (\beta_n/\alpha_n)$. Consequently, the upper and lower 
estimates on the smoothness of $V$ are not necessary for the Riesz and
Schauder basis properties of the root system of $H^P$ and $H^{AP}$, and hence are dictated by the methods used in \cite{DV05}, \cite{KM98}, \cite{Ki09}, \cite{Ku06}, 
\cite{Ma06a}, \cite{Ma06c}, \cite{Ma10}, \cite{MM08}, and \cite{MM10}. The sufficiency 
of the conditions imposed on potentials in the latter papers follows from 
\cite[Theorems\ 1.51, 1.52]{Ma86} (the latter describe the asymptotic behavior of 
$\{\lambda_n^\pm\}\, \text{and}\, \{\mu_j\}$) and our Theorems \ref{t1.2} and \ref{t1.4}. 

Clearly, the direct and inverse spectral approach have their advantages and 
disadvantages, and the interested reader now has the possibility of a choice between 
these two approaches.
\end{remark} 

\begin{remark} \lb{r6.2}
In the special self-adjoint case, where $V\in L^2([0,\pi]; dx)$ is in addition real-valued, standard oscillation theory implies that 
\begin{equation}
\mu_k \in \big[\min(\lambda_{k}^-, \lambda_{k}^+), 
\max(\lambda_{k}^-, \lambda_{k}^+)\big], \quad k \in \bbN,    \lb{5.1}
\end{equation}
(cf., e.g., \cite[Sect.\ 3.4]{Ma86}). Thus, 
\eqref{1.19} and \eqref{1.20} are of course satisfied (as they must be on 
abstract grounds since the system of eigenvectors for any self-adjoint 
operator in $\cH$ with purely discrete spectrum forms an orthonormal basis 
in $\cH$). 
\end{remark}

\begin{remark} \lb{r6.3} 
It follows from Theorem \ref{t1.2}, which is an improved version of the statement 
in \cite[Remark\ 8.10]{GT09}, that if \eqref{1.19} and \eqref{1.20} are satisfied,  
then the root system of {\it every} operator $H(t)$, $t\in[0,2\pi]$, defined in 
\eqref{1.22} contains a Riesz basis. However, we showed in \cite{GT06}, 
\cite{GT09} that this property may not be uniform with respect to $t\in[0,2\pi]$. More precisely, we constructed a potential $V\in L^2([0,\pi]; dx)$ such that 
\eqref{1.19} and \eqref{1.20} are valid, but the family of constants $C=C(t)$ 
in \eqref{4.1} corresponding to the family $H(t)$ is not bounded with respect to 
$t\in[0,2\pi]$. As a result, the corresponding operator $H$ in \eqref{1.25} is not a spectral operator of scalar type in the sense of Dunford (cf.\ 
\cite[Sect.\ XVIII.2]{DS88}). Nevertheless, every Hill operator with a complex-valued locally square-integrable potential is an operator with a separable spectrum as defined by Lyubich and Matsaev \cite{LM60}, \cite{LM62}. 
\end{remark}

\begin{remark} \lb{r6.4} Gasymov showed in \cite{Ga80} (cf.\ also \cite{Ga80a}, \cite{Sh03}) 
 that if
\begin{equation}
V(x)=\sum_{n=1}^\infty\ c_ne^{2inx}, \quad
\{c_n\}_{n\in\bbN}\in \ell^1(\bbN), \;\, x\in [0,\pi],  \lb{5.4}
\end{equation}
then $u_+(\zeta)=\cos(\zeta \pi)$. Thus, in this case
$\sigma(H) =[0,\infty)$ and hence $\bbN_s = \emptyset$. Consequently, the condition \eqref{1.19} (resp.\ \eqref{1.20}) in Theorem \ref{t1.2} is obviously satisfied and hence any operator $H^P$ and $H^{AP}$ associated with a 
potential in the Gasymov class \eqref{5.4} possesses a Riesz basis of root 
vectors in $L^2([0,\pi]; dx)$. To the best of our knowledge, this appears to 
be a new observation. 

However, one notes that the function
\begin{equation}
\frac{\phi(\zeta,\pi)}{u_+^{\bullet}(\zeta)}
=-\frac{2 \zeta \phi(\zeta,\pi)}{\pi\sin(\zeta \pi)}      \lb{5.5}
\end{equation}
is analytic in an open neighborhood of $\sigma(H)$ if and only if
$\phi(\zeta,\pi)=\sin(\zeta \pi)/\zeta$. In the latter case
$u_-(\cdot)\equiv 0$ and $V(x)=0$ for a.e.\ $x\in\bbR$. As discussed in 
\cite{GT06}, \cite{GT09}, this implies that
no smoothness or analyticity conditions imposed on a periodic potential $V$ 
on $\bbR$ can guarantee that a Hill operator $H$ in $L^2(\bbR; dx)$ as 
in \eqref{1.25}, \eqref{1.26} is a spectral operator of scalar type. 

For every operator $H$ in $L^2(\bbR; dx)$ with a nontrivial potential 
\eqref{5.4} on $\bbR$ there
exists at least one integer $n_0\in\N$ such that $\phi(n_0^2,\pi)\neq0$ 
and the point $n_0^2$ is then a spectral singularity. Still, as observed above, 
all potentials in the Gasymov class yield operators $H^P$ and $H^{AP}$ which 
possess a Riesz basis of root vectors in $L^2([0,\pi]; dx)$. 
\end{remark}

\begin{remark} \lb{r6.5} 
In conclusion, we note that there are a variety of ways to fix two vectors 
representing the root subspace corresponding to an eigenvalue 
$\xi^+_{2k} = \xi^-_{2k}$ of algebraic multiplicity two (e.g., for $k\in\bbN$ 
sufficiently large) in a Riesz basis. (For example, any orthonormal pair will 
do for this purpose.) In Section \ref{s3} we demonstrated how to pick such 
a pair based on a study of the singularity structure of the resolvent of the 
operator $H^P$. 
\end{remark}

\medskip

\noindent {\bf Acknowledgments.}
We are indebted to Yura Lyubarskii for helpful comments. V.\ Tkachenko gratefully acknowledges the award of a Miller Scholarship  
from the Department of Mathematics of the University of Missouri, Columbia, USA. 
He is indebted to the department for its great hospitality during his stay in the 
month of October 2010. He also gratefully acknowledges partial 
support by the Israel Science Foundation under grant 125/09. 

\bigskip



\begin{thebibliography}{99}
%
\bi{AG81} N.\ I.\ Akhiezer and I.\ M.\ Glazman, {\it Theory of Linear
Operators in Hilbert Space, Volume I}, Pitman, Boston, 1981.
%
\bi{BG06} V.\ Batchenko and F.\ Gesztesy, {\it On the spectrum of
Schr\"odinger operators with quasi-periodic algebro-geometric KdV
potentials}, J. Analyse Math. {\bf 95}, 333--387 (2005).
%
\bi{Bi86}  B.\ Birnir, {\it Complex Hill's equation and the complex
periodic Korteweg--de Vries equations }, Commun. Pure Appl. Math.
{\bf 39}, 1--49 (1986).
%
\bi{Bi86a} B.\ Birnir, {\it Singularities of the complex Korteweg--de
Vries flows}, Commun. Pure Appl. Math. {\bf 39}, 283--305 (1986).
%
\bi{Bi87} B.\ Birnir, {\it An example of blow-up, for the complex KdV
equation and existence beyond blow-up,} SIAM J. Appl. Math. {\bf 47},
710--725 (1987).
%
\bi{Ch06} T.\ Christiansen, {\it Isophasal, isopolar, and isospectral Schr\"odinger operators and elementary complex analysis}, Amer. J. Math. {\bf 130}, 49--58 (2008). 
%
\bi{CL85} E.\ A.\ Coddington and N.~Levinson,{\it Theory of Ordinary
Differential Equations}, Krieger, Malabar, 1985.
%
\bi{DV05} N.\ Dernek and O.\ A.\ Veliev, {\it On the Riesz basisness of the root 
functions of the nonself-adjoint Sturm--Liouville operator}, Israel J. Math. 
{\bf 145}, 113--123 (2005).
%
\bi{DM02} P.\ Djakov and B.\ Mityagin, {\it Smoothness of Schr\"odinger
operator potential in the case of Gevrey type asymptotics of the gaps}, J.
Funct. Anal. {\bf 195}, 89--128 (2002).
%
\bi{DM03} P.\ Djakov and B.\ Mityagin, {\it Spectral gaps of the periodic
Schr\"odinger operator when its potential is an entire function}, Adv.
Appl. Math. {\bf 31}, 562--596 (2003).
%
\bi{DM03a} P.\ Djakov and B.\ Mityagin, {\it Spectral triangles of
Schr\"odinger operators with complex potentials}, Selecta Math. {\bf 9},
495--528 (2003).
%
\bi{DM06} P.\ Djakov and B.\ Mityagin, {\it Instability zones of 1D periodic 
Schr\"odinger and Dirac operators}, Russ. Math. Surv. {\bf 61}, 
663--766 (2006).
%
\bi{DM09} P.\ Djakov and B.\ Mityagin, {\it Convergence of spectral 
decompositions of Hill operators with trigonometric polynomial potentials}, 
Doklady Math. {\bf 83}, 5--7 (2011).  
%
\bi{DS88}  N.\ Dunford and J.\ T.\ Schwartz, {\it Linear Operators, Part
III: Spectral Operators}, Wiley--Interscience, New York, 1988.
%
\bi{Ea73}  M.\ S.\ P.\ Eastham, {\it The Spectral Theory of Periodic
Differential Equations}, Scottish Academic Press, Edinburgh and
London, 1973.
%
\bi{EF78} W.\ Eberhard and G.\ Freiling, {\it Stone-regul\"are 
Eigenwertprobleme}, Math. Z. {\bf 160}, 139--161 (1978).
%
\bi{EFZ05} W.\ Eberhard, G.\ Freiling, and A.\ Zettl, {\it Sturm--Liouville
problems with singular non-selfadjoint boundary conditions}, Math. Nachr. 
{\bf 278}, 1509--1523 (2005).
%
\bi{Ef07} R.\ F.\ Efendiev, {\it The characterization problem for one class 
of second order operator pencil with complex periodic coefficients}, 
Mosc. Math. J.  {\bf 7}, 55--65, 166 (2007).
%
\bi{Fr84} G.\ Freiling, {\it Zur Vollst\"andigkeit des Systems der 
Eigenfunktionen und Hauptfunktionen irregul\"arer Operatorb\"uschel}, 
Math. Z. {\bf 188}, 55--68 (1984).
%
\bi{FR95} G.\ Freiling and V.\ Rykhlov, {\it On a general class of 
Birkhoff-regular eigenvalue problems}, Differential Integral Eqs. 
{\bf 8}, 2157--2176 (1995).
%
\bi{Ga80} M.\ G.\ Gasymov, {\it Spectral analysis of a class of
second-order non-self-adjoint differential operators}, Funct. Anal.
Appl. {\bf 14}, 11--15 (1980).
%
\bi{Ga80a} M.\ G.\ Gasymov, {\it Spectral analysis of a class of
ordinary differential operators with periodic coefficients}, Sov.
Math. Dokl. {\bf 21}, 718--721 (1980).
%
\bibitem {Ge50}   I.\ M.\ Gel'fand,
{\em Expansion in characteristic functions of an equation with periodic
coefficients}, Doklady Akad Nauk SSSR {\bf 73}, 1117-1120 (1950).
(Russian.)
%
\bi{GM03} F.\ Gesztesy and K.\ A.\ Makarov, {\it $($Modified\,$)$ Fredholm
Determinants for Operators with Matrix-Valued Semi-Separable Integral
Kernels Revisited}, Integral Eqs. Operator Theory {\bf 47},
457--497 (2003). (See also Erratum {\bf 48}, 425--426 (2004) and the
corrected electronic only version in {\bf 48}, 561--602  (2004).)
%
\bi{GT06} F.\ Gesztesy and V.\ Tkachenko, {\it When is a non-self-adjoint Hill operator a spectral operator of scalar type?}, C. R. Acad. Sci. Paris, Ser. I, {\bf 343}, 239--242 (2006).
%
\bi{GT09} F.\ Gesztesy and V.\ Tkachenko, {\it A criterion for Hill operators to be spectral operators of scalar type}, J. Analyse Math. {\bf 107}, 287--353 (2009). 
%
\bi{GW95}  F.\ Gesztesy and R.\ Weikard, {\it Floquet theory revisited},
in {\it Differential Equations and Mathematical Physics}, I.\ Knowles
(ed.), International Press, Boston, 1995, pp.\ 67--84.
%
\bi{GW96} F.\ Gesztesy and R.\ Weikard, {\it Picard potentials
and Hill's  equation on  a torus}, Acta Math. \textbf{176}, 73--107
(1996).
%
\bi{GW98} F.\ Gesztesy and R.\ Weikard, {\it A characterization
of all elliptic  algebro-geometric solutions of the AKNS hierarchy}, Acta
Math. {\bf 181}, 63--108 (1998).
%
\bi{GW98a} F.\ Gesztesy and R.\ Weikard, {\it Elliptic
algebro-geometric solutions of the KdV and AKNS hierarchies -- an
analytic approach}, Bull. Amer. Math. Soc. {\bf 35}, 271--317 (1998).
%
\bi{GGK90} I.\ Gohberg, S.\ Goldberg, and M.\ A.\ Kaashoek, {\it Classes
of Linear Operators, Vol.\ I}, Operator Theory: Advances and 
Applications, Vol.\ 49, Birkh\"auser, Basel, 1990.
%
\bi{GK60} I.\ C.\ Gohberg and M.\ G.\ Krein, {\it The basic propositions on 
defect numbers, root numbers and indices of linear operators},  Amer. Math. 
Soc. Transl. (2) {\bf 13}, 185--264 (1960).
%
\bi{GK69} I.\ C.\ Gohberg and M.\ G.\ Krein, {\it Introduction to the Theory of
Linear Nonselfadjoint Operators}, Translations of Mathematical Monographs, Vol.\ 18, Amer.\ Math.\ Soc., Providence, RI, 1969.
%
\bi{GLO97} A.\ A.\ Gol'dberg, B.\ Ya.\ Levin, I.\ V.\ Ostrovskii, {\it IEntire 
and Meromorphic Functions}, in {\it Complex Analysis I}, Part I, Chapter 4, 
{\it Interpolation by Entire Functions}, co-authored by B.\ Ya. Levin and V.\ A.\ 
Tkachenko, Encyclopaedia of Mathematical Sciences, Vol.\ 85, A.\ A.\ 
Gonchar, V.\ P.\ Havin, and N.\ K.\ Nikolski (eds.), Springer, 1997, pp.\ 67--98.   
%
\bi{GU83} V.\ Guillemin and A.\ Uribe, {\it Hardy functions and the
inverse spectral method}, Commun. PDE {\bf 8}, 1455--1474 (1983).
%
\bi{In56} E.\ L.\ Ince, {\it Ordinary Differential Equations},
Dover, New York, 1956.
%
\bi{KS89} B.\ S.\ Kashin and A.\ A.\ Saakyan, {\it Orthogonal Series}, Translations of Mathematical Monographs, Vol.\ 75, Amer.\ Math.\ Soc., Providence, RI, 1989.
%
\bi{KM98} N.\ B.\ Kerimov and Kh.\ R.\ Mamedov, {\it On the Riesz basis 
property of root functions of some regular boundary value problems}, 
Math. Notes {\bf 64}, 483--487 (1998).
%
\bi{Ke64} G.\ M.\ Kesel'man, {\it On the unconditional convergence of eigenfunction expansions of certain differential operators}, Izv. Vyssh. Uchebn. Zaved. Mat. 
{\bf 39} (2), 82--93 (1964). (Russian.)
%
\bi{Ki09} A.\ A.\ Kirac, {\it Riesz basis property of the root functions of 
non-selfadjoint operators with regular boundary conditions}, Int. J. Math. Anal. (Ruse) {\bf 3}, 1101--1109 (2009).
%
\bi{Ko97} S.\ Kotani, {\it Generalized Floquet theory for stationary
Schr\"odinger operators in one dimension}, Chaos, Solitons \& Fractals
{\bf 8}, 1817--1854 (1997).
%
\bi{Ku06} V.\ M.\ Kurbanov, {\it A theorem on equivalent bases for differential 
operators}, Dokl. Math. {\bf 73}, 11--14 (2006).
%
\bi{Le80} B.\ Ya.\ Levin, {\it Distribution of Zeros of Entire Functions}, 
Translations of Mathematical Monographs, Vol.\ 5, Amer.\ Math.\ Soc., 
Providence, RI, 1980.
%
\bi{Le96} B.\ Ya.\ Levin, {\it Lectures on Entire Functions}, in collaboration 
with Yu.\ Lyubarskii, M.\ Sodin, and V.\ Tkachenko, Transl. 
Math. Monographs, {\bf 150}, Amer. Math. Soc., Providence, RI, 1996.
%
\bi{LM60} U.\ I.\ Lyubi\v c and V.\ I.\ Macaev, {\it On the spectral theory
of linear operators in Banach spaces}, Soviet Math. Dokl.  {\bf 1},
184--186 (1960).
%
\bi{LM62} Ju.\ I.\ Lyubi\v c and V.\ I.\ Macaev, {\it Operators with
separable spectrum}, Mat. Sb. (N.S.)  {\bf 56} (98), 433--468 (1962).
(Russian.)
%
\bi{Ma06} A.\ Makin, {\it On spectral decompositions corresponding to 
non-self-adjoint Sturm--Liouville operators}, Dokl. Math. {\bf 73}, 15--18 
(2006). 
%
\bi{Ma06a} A.\ Makin, {\it Convergence of expansions in the root functions 
of periodic boundary value problems}, Dokl. Math. {\bf 73}, 71--76 (2006).  
%
\bi{Ma06c} A.\ Makin, {\it On periodic boundary value problem for the 
Sturm--Liouville operator}, arXiv:math.SP/0601436.  
%
\bi{Ma06b} A.\ Makin, {\it On the basis property of systems of root functions 
of regular boundary value problems for the Sturm--Liouville operator}, 
Differ. Eq. {\bf 42}, 1717--1728 (2006).
%
\bi{Ma08} M.\ M.\ Malamud, {\it On the completeness of the root vector 
system of the Sturm--Liouville operator with general boundary conditions}, 
Dokl. Math. {\bf 77}, No.\ 2, 175--178 (2008).  
%
\bi{Ma08a} M.\ M.\ Malamud, {\it On the completeness of the system of root 
vectors of the Sturm--Liouville operator with general boundary conditions}, 
Funct. Anal. Appl. {\bf 42}, 198--204 (2008).  
%
\bi{Ma10} Kh.\ R.\ Mamedov, {\it On the basis property in $L_p(0,1)$ of the root 
functions of a class non self adjoint Sturm--Liouville operators}, European J. Pure 
Appl. Math. {\bf 3}, 831--838 (2010). 
%
\bi{MM08} Kh.\ R.\ Mamedov and H.\ Menken, {\it On the basisness in 
$L_2(0,1)$ of the root functions in not strongly regular boundary value 
problems},  Eur. J. Pure Appl. Math. {\bf 1},  no. 2, 51--60 (2008).
%
\bi {Ma86} V.\ A.\ Marchenko, {\it Sturm--Liouville operators and
applications}, Birkh\"auser, Basel, 1986.
%
\bi{Me77} N.\ N.\ Meiman, {\it The theory of one-dimensional Schr\"odinger
operators with a periodic potential}, J. Math. Phys. {\bf 18}, 834--848
(1977).
%
\bi{MM10} H.\ Menken and Kh.\ R.\ Mamedov, {\it Basis property in 
$L_p(0,1)$ of the root functions corresponding to a boundary-value 
problem}, J. Appl. Funct. Anal. {\bf 5}, 351--356 (2010).
%
\bi{Mi62} V.\ P.\ Miha{\u i}lov, {\it Riesz bases in $L_2(0,1)$}, Sov. Math. Dokl. 
{\bf 3}, 851--855 (1962). 
%
\bi{Mi06} A.\ Minkin, {\it Resolvent growth and Birkhoff-regularity}, J. Math. 
Anal. Appl. {\bf 323}, 387--402 (2006).
%
\bi{Mi10} A.\ Minkin, {\it Regularity of dissipative differential operators}, 
arXiv:math/9909092v2, Jan. 2010.  
%
\bi{Na67} M.\ A.\ Naimark, {\it Linear Differential Operators, Part I},
Ungar, New York, 1967.
%
\bi{PT88}  L.\ A.\ Pastur and V.\ A.\ Tkachenko, {\it Spectral theory
of Schr\"odinger operators with periodic complex-valued potentials},
Funct. Anal. Appl. {\bf 22}, 156--158 (1988).
%
\bi{PT91} L.\ A.\ Pastur and V.\ A.\ Tkachenko, {\it An inverse
problem for a class of one-dimensional Schr\"odinger operators with a
complex periodic potential}, Math. USSR Izv. {\bf 37}, 611--629 (1991).
%
\bi{PT91a} L.\ A.\ Pastur and V.\ A.\ Tkachenko, {\it Geometry of the
spectrum of the one-dimensional Schr\"odinger equation with a periodic
complex-valued potential}, Math. Notes {\bf 50}, 1045--1050 (1991).
%
\bi{RS78} M.\ Reed and B.\ Simon, {\it Methods of Modern
Mathematical Physics. IV: Analysis of Operators}, Academic Press, New
York, 1978.
%
\bi{Ro63} F.\ S.\ Rofe-Beketov, {\it The spectrum of non-selfadjoint
differential operators with periodic coefficients}, Sov. Math. Dokl.
{\bf 4}, 1563--1566 (1963).
%
\bi{ST96} J.-J.\ Sansuc and V.\ Tkachenko, {\it Spectral
parametrization of non-selfadjoint Hill's operators}, J. Diff. Eq.
{\bf 125}, 366--384 (1996).
%
\bi{ST96a} J.-J.\ Sansuc and V.\ Tkachenko, {\it Spectral properties of
non-selfadjoint Hill's operators with smooth potentials}, in {\it Algebraic
and Geometric Methods in Mathematical Physics}, A.\ Boutel de
Monvel and V.\ Marchenko (eds.), Kluwer, Dordrecht, 1996, pp.\ 371--385.
%
\bi{ST97} J.-J.\ Sansuc and V.\ Tkachenko, {\it Characterization of the
periodic and  antiperiodic spectra of nonselfadjoint Hill's operators,}
in {\it New  Results in Operator Theory and its Applications},
I.\ Gohberg and  Yu.\ Lubich eds.), Operator Theory: Advances and
Applications {\bf 98},  Birkh\"auser, Basel, 1997, pp.\ 216--224.
%
\bi{Se60} M.\ I.\ Serov, {\it Certain properties of the spectrum of a
non-selfadjoint differential operator of the second order}, Sov. Math.
Dokl. {\bf 1}, 190--192 (1960).
%
\bi{Sh03} K.\ C.\ Shin, {\it On half-line spectra for a class of
non-self-adjoint Hill operators}, Math. Nachr. {\bf 261--262}, 171--175
(2003).
%
\bi{Sh04} K.\ C.\ Shin, {\it Trace formulas for non-self-adjoint
Schr\"odinger operators and some applications}, J. Math. Anal. Appl.
{\bf 299}, 19--39 (2004).
%
\bi{Sh04a} K.\ C.\ Shin, {\it On the shape of spectra for non-self-adjoint
periodic Schr\"odinger operators}, J. Phys. A {\bf 37}, 8287--8291 (2004).
%
\bi{SS07} E.\ A.\ Shiryaev and A.\ A.\ Shkalikov, {\it Regular and completely 
regular differential operators}, Math. Notes {\bf 81}, 566--570 (2007).
%
\bi{Sh76} A.\ A.\ Shkalikov, {\it The completeness of the eigenfunctions and associated functions of an ordinary differential operator with 
irregular-separated boundary conditions}, Funct. Anal Appl. {\bf 10}, 
305--316 (1976).
%
\bi{Sh79} A.\ A.\ Shkalikov, {\it On the basis problem of the eigenfunctions of an ordinary differential operator}, Russ. Math. Surv. {\bf 34}, no.\ 5, 249--250 
(1979).
%
\bi{Sh82} A.\ A.\ Shkalikov, {\it The basis problem of the eigenfunctions of 
ordinary differential operators with integral boundary conditions}, Moscow 
Univ. Math. Bull. {\bf 37}, no.\ 6, 10--20 1982.
%
\bi{SV09} A.\ A.\ Shkalikov and O.\ A.\ Veliev, {\it On the Riesz basis 
property of eigen- and associated functions of periodic and anti-periodic 
Sturm--Liouville problems}, Math. Notes {\bf 85}, 647--660 (2009).
%
\bi{Ti58} E.\ C.\ Titchmarsh, {\it Eigenfunction Expansions
associated with Second-Order Differential Equations, Part II}, Oxford
University Press, Oxford, 1958.
%
\bi{Tk64} V.\ A.\ Tkachenko, {\it Spectral analysis of the one-dimensional
Schr\"odinger operator with periodic complex-valued potential}, Sov.
Math. Dokl. {\bf 5}, 413--415 (1964).
%
\bi{Tk92} V.\ A.\ Tkachenko, {\it Spectral analysis of a nonselfadjoint
Hill operator}, Sov. Math. Dokl. {\bf 45}, 78--82 (1992).
%
\bi{Tk94} V.\ A.\ Tkachenko, {\it Discriminants and Generic Spectra
of Nonselfadjoint Hill's Operators}, Adv. Sov. Math. {\bf 19}, 41--71
(1994).
%
\bi{Tk96} V.\ A.\ Tkachenko, {\it Spectra of non-selfadjoint
Hill's operators and a class of Riemann surfaces}, Ann. Math. {\bf
143}, 181--231 (1996).
%
\bi{Tk01} V.\ Tkachenko, {\it Characterization of Hill operators with
analytic potentials}, Integral equ. oper. theory {\bf 41}, 360--380
(2001).
%
\bibitem {Tk02} V.\ Tkachenko, {\em Non-selfadjoint Sturm-Liouville
operators with multiple spectra}, in {\it Interpolation Theory,
Systems Theory and Related Topics}, D. Alpay, I. Gohberg, V.
Vinnikov (eds.), Operator Theory: Advances and Applications, {\bf 134},
403--414 (2002).
%
\bi{Ve80} O.\ A.\ Veliev, {\it The one-dimensional Schr\"odinger operator
with a periodic complex-valued potential}, Sov. Math. Dokl. {\bf 21},
291--295 (1980).
%
\bi{Ve83} O.\ A.\ Veliev, {\it Spectrum and spectral singularities of
differential operators with complex-valued periodic coefficients}, Diff.
Eqs. {\bf 19}, 983--989 (1983).
%
\bi{Ve86} O.\ A.\ Veliev, {\it Spectral expansions related to
non-self-adjoint differential operators with periodic coefficients}, Diff.
Eqs. {\bf 22}, 1403--1408 (1986).
%
\bi{Ve06} O.\ A.\ Veliev, {\it Spectral expansion for a nonselfadjoint periodic differential operator}, Russ. J. Math. Phys. {\bf 13}, 101--110 (2006).
%
\bi{Ve10} O.\ A.\ Veliev, {\it On the nonself-adjoint ordinary differential 
operators with periodic boundary conditions}, Israel J. Math. {\bf 176}, 
195--207 (2010). 
%
\bi{VT02} O.\ A.\ Veliev and M.\ Toppamuk Duman, {\it The spectral
expansion for a nonself-adjoint Hill operator with a locally integrable
potential}, J. Math. Anal. Appl. {\bf 265}, 76--90 (2002).
%
\bi{VK07} O.\ A.\ Veliev and A.\ A.\ Kirac, {\it On the nonself-adjoint 
differential operators with the quasiperiodic boundary conditions}, Int. 
Math. Forum {\bf 2}, 1703--1715 (2007).   
%
\bi{Wa77} P.\ W.\ Walker, {\it A nonspectral Birkhoff-regular differential operator}, 
Proc. Amer. Math. Soc. {\bf 66}, 187--188 (1977). 
%
\bibitem{We98} R.\ Weikard, {\it On Hill's equation with a singular
complex-valued potential}, Proc. London Math. Soc. {\bf 76}, 603-633
(1998).
%
\bibitem{We98a} R.\ Weikard, {\it On a theorem of Hochstadt}, Math.
Ann. {\bf 311}, 95-105 (1998).
%
\bi{Yo01} R.\ M.\ Young, {\it An Introduction to Nonharmonic Fourier Series}, 
ref. 1st ed., Academic Press, San Diego, 2001. 
%
\bibitem{Zh69} V.\ A.\ Zheludev, {\it Perturbations of the spectrum of the
Schroedinger operator with a complex periodic potential}, in {\it Topics
in Mathematical Physics {\bf 3}, Spectral Theory}, M. Sh. Birman (ed.),
Consultants Bureau, New York (1969), pp.\ 25--41.
%
\bi{Zy90} A.\ Zygmund, {\it Trigonometric Series}, 2nd ed., Vol.\ I \& II combined, 
Cambridge University Press, Cambridge, 1990. 
%
\end{thebibliography}
\end{document}